\documentclass[10pt, a4paper]{amsart}  %envcountsame

\usepackage{amsmath}
\usepackage{amsfonts,amssymb}
\usepackage{times}
\usepackage{url}

\usepackage{xypic}
\usepackage{slashed} 
\usepackage{graphicx}

\newcommand{\set}[1]{\left\{#1\right\}}
\newcommand{\R}{\mathbb R}
\newcommand{\C}{\mathbb C}
\newcommand{\Z}{\mathbb Z}
\newcommand{\Zz}{\Z/2\Z}
\newcommand{\bbone}{\mathbf{1}}
\newcommand{\eps}{\varepsilon}

\DeclareMathOperator{\alim}{a-lim}

\newcommand{\Sc}{\boldsymbol \nabla}

\newcommand{\Ahat}{\hat{A}}
\newcommand{\detline}[1]{\mathrm{det}({#1})}
\DeclareMathOperator{\sTr}{Tr}
\newcommand{\Tr}{\sTr_0}
\newcommand{\Sp}{\mathbb{S}}

\newcommand{\dslash}{\slashed{D}}
\DeclareMathOperator{\spec}{spec}

\newcommand{\Det}[1]{\mathrm{det}\,#1}

\newcommand{\bb}[1]{{\left(#1\right)}}
\newcommand{\Ll}{{L^2}}

\newcommand{\Forms}{\Omega}
\newcommand{\stimes}{\otimes}
\DeclareMathOperator{\End}{End}
\DeclareMathOperator{\Hom}{Hom}
\newcommand{\even}{\mathrm{ev}}
\newcommand{\odd}{\mathrm{odd}}
\DeclareMathOperator{\sCh}{ch}
\newcommand{\ud}{\mathrm{d}}
\newcommand{\secfun}{\mathrm{I\!\!\:\!\:I}}
\newcommand{\Curv}[1]{\Omega^{#1}}
\newcommand{\curv}[1]{\mathrm{Curv}\left({#1}\right)}

\DeclareMathOperator{\dist}{dist}

\DeclareMathOperator{\Cliff}{Cl}
\newcommand{\Dirac}{\mathcal{D}}

\newcommand{\Hol}[2]{\mathrm{Hol}_{#1}\,#2}

\DeclareMathOperator{\Index}{index}

\newcommand{\extm}{e}
\DeclareMathOperator{\LIM}{LIM}
\newcommand{\intm}{\imath}

\newcommand{\isarrow}{\xrightarrow{\;\sim\;} }
\DeclareMathOperator{\id}{id}

\numberwithin{equation}{section}

\newtheorem{theorem}[equation]{Theorem}

\newtheorem{lemma}[equation]{Lemma}
\newtheorem{proposition}[equation]{Proposition}
\theoremstyle{definition}
\newtheorem{definition}[equation]{Definition}
\theoremstyle{remark}
\newtheorem{remark}[equation]{Remark}

\begin{document}
\title{Superconnections and Index Theory}

\author{Alexander Kahle}
\email{{kahle@uni-math.gwdg.de}}
\address{Mathematisches Institut, Universit\"at G\"ottingen, Bunsenstr. 3-5, 37073 G\"ottingen, GERMANY.}

\begin{abstract}
We investigate index theory in the context of Dirac operators coupled to superconnections. In particular, we prove a local index theorem for such operators, and for families of such operators. We investigate $\eta$-invariants and prove an APS-theorem, and construct a geometric determinant line bundle for families of such operators, computing its curvature and holonomy in terms of familiar index theoretic quantities.
\end{abstract}

\maketitle

\section{Introduction}
This paper reports on work done to extend the framework of geometric index theory to Dirac operators coupled to superconnections. The goal of this paper is to present a summary of the results in \cite{K1}, highlighting the main theorems and constructions, and sketching their proofs.

The notion of a superconnection was introduced to mathematics by Daniel Quillen \cite{Q1} as a generalisation of the notion of a connection to the category of $\Zz$-graded vector bundles. It has found wide application, especially in physics. 

We recall that a superconnection on a smooth $\Zz$-graded vector bundle $V\to X$ is an odd derivation on $\Forms(X;V)^\bullet$ as a module over the dg-algebra $\Forms^\bullet(X)$. Concretely, every superconnection $\Sc$ may be written
\begin{equation}\label{eq:introscdecomp}
\Sc=\nabla+\sum_i\omega_i,
\end{equation}
where $\nabla$ is a connection on $V$, and each $\omega_i$ is an odd element of $\Forms^i(X;\End(V))$.

One may form Dirac operators out of superconnections in a manner entirely analogous to the construction of Dirac operators from ordinary connections. As with connections, there is a notion of unitarity for superconnections, and Dirac operators coupled to unitary superconnections are formally self-adjoint. On compact manifolds they are elliptic first-order differential operators, and share many of the analytic properties of Dirac operators coupled to connections. 

One appealing aspect of Dirac operators coupled to superconnections is their behaviour for families. We will see in Sec.~\ref{sec:families} that, for an appropriate geometric family\footnote{The appropriate notion of family in this case is a so-called Riemannian map: see Def.~\ref{def:riemannmap}.} $\pi:X\to Y$ with compact and spin fibres, that one may define the push-forward of a complex hermitian $\Zz$-graded vector bundle  $V\to X$ with superconnection $\Sc$, which will itself be a complex hermitian $\Zz$-graded vector bundle $\pi_!V\to Y$ with superconnection $\pi_!\Sc$. This construction is strictly functorial: if we have a further Riemannian map $\rho:Y\to Z$ with compact and spin fibres, then $\rho_!(\pi_!V,\pi_!\Sc)\cong(\rho\circ\pi)_!(V,\Sc)$. In particular, setting $Z=\{\mathrm{pt}\}$, one obtains the equality $\Dirac_Y(\pi_!\Sc)=\Dirac_X(\Sc)$.

This behaviour makes $\Zz$-graded vector bundles with superconnections appealing objects to take as cycles in models for differential $K$-theory\footnote{Differential cohomology theories are functors that extend cohomology theories by differential forms, and classes in a differential cohomology theory thus capture not only global topological aspects captured by the original cohomology theory, but also are sensitive to local information from the differential forms. Differential cohomology theories have found important application in physics: essentially classes classify gauge-equivalence classes of fields in abelian gauge theories. In particular, gauge equivalent classes of Ramond-Ramond fields are supposed to  be classified by differential $K$-theory. A good introduction to the subject from a physical point of view is \cite{FrDirac}. Differential $K$-theory also has interesting applications to index theory: many secondary invariants find natural homes there. Some basic mathematical references for the subject are \cite{CheSi}, \cite{Del}, both of which are foundational papers; \cite{HoSi}, where differential versions of generalised cohomology theories are constructed; and \cite{BuSchi}, which gives a construction of differential $K$-theory in terms of natural geometric cycles.}. As Quillen notes \cite{Q1} one may model both even and odd degree $K$-theory classes in terms of $\Zz$-graded vector bundles with superconnection, suggesting that one might take these as cycles in a geometric model for differential $K$-theory. Moreover, $\Zz$-graded vector bundles with superconnections have a natural notion of support associated to them, and their Chern character is sharply peaked about this support. One might thus imagine that a model of differential $K$-theory based on $\Zz$-graded vector bundles and superconnections would also naturally capture classes with support. In this context, the work described in this paper presents the basic results in local index theory needed when seeking to construct such a model.

 In this paper we extend many results of local index theory to Dirac operators coupled to superconnections:  we prove a local index theorem for these operators (first proved by Getzler in \cite{G2}), and extend this to families, obtaining a generalisation of Bismut's local families index theorem. We investigate the case when the manifold has a boundary, and prove an APS-like theorem. Finally, we construct a geometric determinant line bundle associated to a family of such operators, with metric, connection, and section. While our constructed is related to the general construction in \cite{BF1, BF2}, the connection obtained is different, being constructed in such a way as to obtain formulas for the curvature and holonomy of the line bundle reminiscent of those in \cite{DF1}.

Section \ref{sec:prelim} reviews the language of superconnections and sets notation. The first theorem of the paper is in Sec.~\ref{sec:localindex}: a local index theorem for Dirac operators coupled to superconnections on compact manifolds, first proved by Getzler \cite{G2}. His proof relies on stochastic techniques and the Feynman-Kac formula -- we provide a new proof based on elementary analysis of the heat kernel.

In order to state the local index theorem we need to set some notation, and introduce an $\R^\times$-action on the space of superconnections. \emph{This action will turn out to be the fundamental new ingredient that will pervade the study of the index theory of Dirac operators coupled to superconnections.} Let $X$ be a compact and spin Riemannian manifold, and $V\to X$ be a finite dimensional, complex and hermitian $\Zz$-graded vector bundle, and $\Sc$ is a unitary superconnection on it. In terms of the decomposition of $\Sc$ (Eq.~\ref{eq:introscdecomp}) it is given by
\begin{equation}\label{eq:introraction}
\Sc^s=\nabla +\sum_i|s|^{(1-i)/2}\omega_i,
\end{equation}
where $s\in\R^\times$. The heat semigroup associated to $\Dirac(\Sc^s)$, $e^{-t\Dirac(\Sc^s)^2}$,
is smoothing for $t>0$, and thus has an integral kernel associated to it, which we denote by $p_{t,s}(x,y)$ ($t\in\R^{>0}$ is time, $s\in\R^\times$ is the parameter on the superconnection, $x$ and $y\in X$). The local index theorem of Getzler then is as follows.
\begin{theorem}\label{th:introlocalindex} Let $X$ be a compact, spin and Riemannian manifold, with finite dimensional complex and hermitian $\Zz$-graded vector bundle $V\to X$ with unitary superconnection $\Sc$. Then
\begin{equation}\label{eq:introindex}
\lim_{t\to 0}\sTr p_{t,1/t}(x,x)\;\ud x=(2\pi i)^{-\dim X/2}\left[\Ahat(\Curv{X})\sCh\Sc\right]_\bb{n},
\end{equation}
where $p_{t,s}(x,y)$ is the integral kernel associated with the heat semigroup $\exp[-t\Dirac(\Sc^s)^2]$.\footnote{Throughout the paper, ``$\sTr$" denotes the supertrace. The ungraded trace is denoted by ``$\Tr$".}
\end{theorem}
The $\R^\times$-action on superconnections is \emph{crucial} in the statement -- the small $t$ limit of $\sTr p_{t,s}(x,x)$ does not in general exist, and it is only when the parameter is allowed to vary simultaneously that one gets convergence. This however introduces a formidable technical difficulty: understanding the small $t$ behaviour of $p_{t,1/t}(x,x)$. As remarked, Getzler does so using stochastic techniques. We take an entirely different, and to our knowledge, novel approach (the idea is sketched out for Dirac operators coupled ordinary connections in Freed's online notes on index theory\cite{FrN}). We use parabolic scaling to relate the behaviour  the heat kernel at small time with the superconnection appropriately scaled to the behaviour of the heat kernel at time 1 with an unscaled superconnection but with space blown up. By excising a neighbourhood of the point of interest, and gluing it into $\R^{\dim X}$ we are able to understand the behaviour of the heat kernel under blow up, and thus the behaviour of the heat kernel at small times coupled to the parameter on the superconnection.

Section \ref{sec:boundary} examines Dirac operators coupled to superconnections on manifolds with boundary. We define an $\eta$-invariant modulo $\Z$ for Dirac operators coupled to superconnections and proves an APS-theorem for the invariant. Again, the presence of the $\R^\times$-action on superconnections in Th.~\ref{th:introlocalindex} presents difficulties: in order to have a hope of recovering a recognisable APS-theorem we need to take the scaling into account in defining $\eta$-invariants. We do so by defining the invariants geometrically: the $\eta$-invariant of a Dirac operator coupled to a superconnection is defined in terms of the $\eta$-invariant of a Dirac operator coupled to an ordinary connection and the Chern-Simons form between this connection and the superconnection in question. Our guiding principle in defining the $\eta$-invariant this way was to obtain an invariant that satisfied an APS theorem; that is, we wanted the $\eta$-invariant on the boundary to be computed by the integral of the $\Ahat(\Omega)\sCh(\Sc)$ differential form on the interior, modulo integers. 

Bismut and Cheeger \cite{BisChe} also examine $\eta$-invariants coupled to superconnections, but define them differently, essentially as the renormalised count in the difference between the number of positive and negative eigenvalues of the Dirac operator coupled to the superconnection. Their definition is useful for their purposes (namely, as a test case to understand $\eta$-invariants for families of Dirac operators) but it does not take the $\R^\times$-action (Eq.~\ref{eq:introraction}) into account, and so their $\eta$-invariants do not satisfy an APS theorem in the form we desire. This is a general issue in investigating Dirac operators coupled to superconnections: one is forced to choose whether one wants definitions that follow the spirit of the spectral definitions, or whether one wants to recover theorems in terms of geometric quantities like the Chern character. 

Section \ref{sec:families} discusses families of Dirac operators coupled to superconnections. We review the geometry of Riemannian families, and associate to a family of Dirac operators coupled to superconnections a push-forward superconnection that is strictly functorial. This pushforward may be thought to represent the element in $K$-theory given by the families index theorem. We compute the Chern character of the pushed-forward superconnection. This section extends Bismut's work for families of Dirac operators coupled to connections \cite{Bis}. As with the local index theorem, the $\R^\times$-action on superconnections enters crucially into the statement of the families index theorem for Dirac operators coupled to superconnections (Th.~\ref{th:families}). Our proof of Th.~\ref{th:families} follows that of the local index theorem outlined in Sec.~\ref{sec:localindex}.

Section \ref{sec:determinant} defines a geometric line bundle, the determinant line bundle, associated to a family of Dirac operators coupled to superconnections, and computes its curvature and holonomy. 

Again we are faced with a dichotomy: we may either ignore the $\R^\times$-action on superconnections or include it. In the former case, the work in \cite{BF1,BF2} goes through directly to give a geometric determinant line bundle, but its curvature and holonomy are not computed by familiar quantities, but simply by certain terms in the small $t$ asymptotic expansion of the heat kernel associated with the family of Dirac operators. Our goal is to have the curvature and holonomy of the line bundle computed by familiar geometric quantities, and we must thus take the $\R^\times$-action into account. In order to understand how to do so, we begin Sec.~\ref{sec:determinant} with an examination of a simpler finite dimensional case.

Our determinant line bundle is constructed as follows. We use the construction of Quillen \cite{Q2} (for the determinant line bundle associated to a family of $\bar{\partial}$ operators) and Bismut and Freed \cite{BF1,BF2} (for the determinant line bundle associated to a family of Dirac operators coupled to connections) to define the line bundle with section and metric. It is in constructing the connection where the novelty of our construction appears. We modify the connection obtained from the construction of \cite{BF1} in a canonical manner, in order to ensure that its curvature is computed by
\[
(2\pi i)^{-\frac{1}{2}\dim{X/Y}}\left[\pi_*\Ahat(\Omega^{X/Y})\sCh{\Sc}\right]_\bb{2}.
\] 
We show that the holonomy of this connection is computed in terms of the $\eta$-invariant reduced mod $\Z$, in a manner similar to that in \cite{BF1}.

Determinant line bundles occur in Quantum Field Theory when the theory contains fermionic fields. In this case integration over the fermions of the exponentiated action gives rise to a ``determinant of Dirac" term that may be interpreted as lying in a determinant line bundle. When the action contains a Dirac operator coupled to a superconnection, one should obtain a section of the determinant line bundle for a family of Dirac operators coupled to superconnections. Indeed, Lukic and Moore \cite{LM} investigated the quantum integrand in 11-dimensional supergravity  (building on work by Freed and Moore \cite{FM}) and interpreted the integrand in precisely this way. Questions about anomalies become questions about the curvature and holonomy of the determinant line bundle: the former is the ``local" anomaly, the latter the ``global" anomaly. In their paper, Moore and Lukic found that fluxes had to be introduced for anomaly cancellation. However, they used a ``spectral" definition of the determinant line bundle for a family of Dirac operators coupled to superconnections, which does not take the $\R^\times$-action on superconnections into account -- it is natural to wonder whether the ``geometric" determinant line bundle presented here is perhaps more appropriate, and investigate whether anomalies might actually cancel directly in this setting.

To summarise: the essential novelty that pervades the index theory of Dirac operators coupled to superconnections is the presence of the $\R^\times$-action (Eq.~\ref{eq:introraction}). It presents numerous technical difficulties, both in the proof of the index theorems, and in the definition of secondary invariants and constructions associated to these operators. However, taken properly into account the basic theorems and formulas for Dirac operators coupled to superconnections end up very similar to those for Dirac operators coupled to ordinary connections.
\subsection*{Acknowlegements} The author is very grateful to Dan Freed for his patience and advice, and to the referee for the careful reading of this paper, and the consequent insightful and helpful remarks.
\section{Preliminaries}\label{sec:prelim}
We briefly review the language $\Zz$-graded vector bundles and Quillen's notion of superconnections \cite{Q1}. A comprehensive and beautiful discussion may be found in \cite{DelMor}.

A $\Zz$-graded vector space $V$ is a vector space along with a decomposition $V=V^0\oplus V^1$. The zero summand in the decomposition is called \emph{even}, and the one summand, \emph{odd}. Associated to a $\Zz$-graded vector bundle is a canonical endomorphism, the \emph{grading} endomorphism, $\eps:V\to V$, which acts as the (negative of the) identity on even (odd) elements. The tensor product of two $\Zz$-graded vector spaces is naturally $\Zz$-graded, the grading being given by the tensor product of their respective grading homomorphisms. There is a canonical isomorphism $V\stimes W\isarrow W\stimes V$ ($V$ and $W$ are $\Zz$-graded vector spaces) given by 
\[
v\stimes w\mapsto (-1)^{|v||w|}w\stimes v,
\]
where $v\in V$, $w\in W$ and $|\cdot|$ gives the degree of a homogeneous element of a $\Zz$-graded vector space. 

A $\Zz$-graded algebra is an algebra along with a $\Zz$-grading, such that the grading endomorphism is an algebra homomorphism. The basic example of these is furnished by the endomorphism algebra $\End(V)$ of a $\Zz$-graded vector bundle $V$. Its grading is induced from the identification $\End(V)\cong V\stimes V^*$: the even elements are those that preserve the grading, the odd ones those that reverse it.  Care must be taken with the tensor product of two $\Zz$-graded algebras. Suppose $A$ and $B$ are both $\Zz$-graded algebras. Then $A\stimes B$ is again a $\Zz$-graded algebra, with the multiplication defined (on homogeneous elements) by
\[
(a\stimes b)\cdot(a'\stimes b')=(-1)^{|a'||b|}(a\cdot a')\stimes(b\cdot b'),
\]
where $a$, $a'\in A$ and $b$, $b'\in B$.
The \emph{commutator} on a $\Zz$-graded algebra $A$ is defined by
\[
[\cdot,\cdot]:a\stimes b\to a\cdot b-(-1)^{|a||b|}b\cdot a,
\]
where $a$, $b\in A$. Pairs of elements that commutator vanishes on are said to commute. A trace on a $\Zz$-graded algebra is a homomorphism to the ground field that vanishes on the image of the commutator. The algebra of endomorphisms of a $\Zz$-graded algebra has a canonical trace defined by
\[
\sTr:\End(V)\isarrow V\stimes V^*\isarrow V^*\stimes V\rightarrow k,
\]
where $k$ is the ground field, and the last map is the pairing between $V$ and $V^*$. In terms of the ungraded trace $\Tr$ on $\End(V)$ this is given by $\sTr=\Tr\circ\eps$. 

One may of course extend the above discussion to families and talk of bundles of $\Zz$-graded vector spaces and algebras. We will always take these to be smooth, and either real or complex. Associated to a smooth manifold $X$ there is a canonical differential $\Zz$-graded algebra, the algebra of differential forms $\Forms^\bullet(X)$ with the de Rham differential $\ud$, where the grading is the cohomological degree reduced mod two. Given a $\Zz$-graded vector bundle $V\to X$, one may form the $\Forms^\bullet(X)$-module of forms with values in $V$,
\[
\Forms^\bullet(X;V)\equiv\Forms^\bullet(X)\stimes_{\Forms^0(X)}V.
\]
This is a left module for the $\Zz$-graded algebra of $\End(V)$-valued differential forms $
\Forms^\bullet(X;\End(V))
$. 

Quillen \cite{Q1} introduced the notion of the \emph{superconnection} as a generalisation of connections to the category of $\Zz$-graded vector bundles. A superconnection $\Sc$ on a $\Zz$-graded vector bundle $V$ is an \emph{odd} derivation on $\Forms^\bullet(X;V)$. As such, it obeys
\[
\Sc(\omega\stimes v)=(\ud\omega)\stimes v+(-1)^{|\omega|}\omega\stimes\Sc v,
\]
where $\omega\in\Forms^\bullet(X)$ is homogeneous, and $v\in V$. The space of superconnections on a vector bundle is affine, and modelled on $\Forms(X;\End(V))^\odd$: every superconnection admits a decomposition
\begin{equation}\label{eq:scdecomp}
\Sc=\nabla+\sum_i\omega_i,
\end{equation}
where $\Sc$ is a superconnection, $\nabla$ is an ordinary connection, and the $\omega_i\in\Forms^i(X;\End(V))$ are odd in the total grading. One associates a curvature to a superconnection by
\[
\curv{\Sc}\equiv\Sc^2.
\]
This is purely algebraic and even, that is $\curv{\Sc}\in\Forms(X;\End(V))^\even$. To the data of a $\Zz$-graded vector bundle $V$ with superconnection $\Sc$, Quillen associates a Chern character form by
\[
\sCh(\Sc)=\sTr\exp[-\curv{\Sc}].
\]
This is a closed differential form and the class it represents in de Rham cohomology is independent of the choice of superconnection on $V$. Indeed, the useful transgression formula
\begin{equation}\label{eq:transgression}
\frac{\partial}{\partial t}\sCh(\Sc_t)=-\ud\sTr[e^{-\Sc_t^2}\dot{\Sc}_t],
\end{equation}
where $\Sc_t$ is a smooth one parameter family of superconnections on $V$, and `` $\dot{}$ " denotes differentiation with respect to the parameter, shows this.

There is a crucial action of $\R^\times$ on the space of superconnections on $V$. In terms of the decomposition in Eq.\ref{eq:scdecomp}
it is given by
\begin{equation}\label{eq:raction}
\Sc^s=\nabla+\sum_i |s|^{(1-i)/2}\omega_i.
\end{equation}
\subsection{Clifford bundles and Dirac operators}
Let $X$ now be Riemannian and spin. Let $\Cliff(X)\to X$ be the bundle of Clifford algebras, with fibre at a typical point $x\in X$ being $\Cliff(T^* X, g)$, where $g$ is the metric on $X$. We denote by $\Sp(X)\to X$ the bundle of \emph{$\Zz$-graded} spinors\footnote{By this we mean the following. For $X$ even dimensional, we recall that the bundle of spinors has a natural $\Zz$-grading. For $\dim(X)$ odd, we take $\Sp(X)$ to be two copies of the ungraded bundle of spinors, and declare that $\Cliff(X)$ acts in a graded sense.} on $X$. We denote the left action of $\Cliff(X)$ on $\Sp(X)$ by $c(\cdot)$. Every \emph{$\Zz$-graded} $\Cliff(X)$ module $M\to X$ may be (non-canonically) decomposed as $M\cong\Sp(X)\stimes V$, where $V$ is a $\Zz$-graded vector bundle.

We are now in a position to define \emph{Dirac operators} associated to superconnections.
\begin{definition}\label{def:dirac} Let $X$ be a Riemannian and spin manifold. Let $V\to X$ be a $\Zz$-graded vector bundle, with superconnection $\Sc$. The \emph{Dirac operator} associated to this data is the first order linear differential operator defined by the sequence
\[
\Dirac(\Sc):\xymatrix@1{ \Gamma\left(\Sp(X)\stimes V \right)\ar[rrr]^-{(\nabla^\Sp\stimes \bbone)\oplus(\bbone\oplus\Sc)} &&&\Forms^\bullet(X;\,\Sp(X)\stimes V )\ar[r]^(.55){c(\cdot)} & \Gamma\left(\Sp(X)\stimes V \right) }.
\]
\end{definition}
There is a notion of unitarity for superconnections on complex and hermitian $\Zz$-graded vector bundles. Let $V\to X$ be such a bundle, and $\Sc$ be a superconnection on $V$. Locally we may write
\[
 \Sc=\ud+\sum_{i,j}\omega_{ij}\stimes E_{ij}
\]
where $\omega_{ij}\in\Forms^i(X)$, $E_{ij}\in\End^j(V)$, and $i+j=1\mod 2$. Then $\Sc$ is \emph{unitary} if
\[
 E_{ij}\text{ is }\begin{cases}
                          \text{hermitian}&\text{if $i=0,\,3\mod 4$},\\
                          \text{anti-hermitian}&\text{if $i=1,\,2\mod 4$}.
                         \end{cases}
\]
This restricts to the usual notion of unitary if $\Sc$ is a connection.

If $V\to X$ is complex and hermitian, and $\Sc$ is unitary, then the Dirac operator $\Dirac(\Sc)$ is formally self-adjoint. In terms of the decomposition of $\Gamma(\Sp(X)\stimes V)$ as a $\Zz$-graded vector bundle, $\Dirac(\Sc)$ is then block off-diagonal as follows:
\begin{equation}\label{eq:dslashdef}
\Dirac(\Sc)=\begin{pmatrix}
& \dslash(\Sc)^*\\
\dslash(\Sc) &
\end{pmatrix}.
\end{equation}
\section{A local index theorem}\label{sec:localindex}
Dirac operators coupled to superconnections (Def.~\ref{def:dirac}) are first order differential operators, and on compact manifolds, they are elliptic. They thus have a well defined index associated to them, which we wish to calculate. Henceforth we have the following setting: $X$ is a closed, Riemannian and spin manifold, $V\to X$ is a complex and hermitian $\Zz$-graded vector bundle, and $\Sc$ is a unitary superconnection on $V$. The Atiyah-Singer index theorem \cite{ASI} then shows
\begin{equation}\label{eq:indextheorem}
\Index \dslash(\Sc)=(2\pi i)^{-\dim X/2}\int_X\Ahat(\Curv{X})\sCh(\Sc).
\end{equation}
Our first task will be to find a local refinement for this statement, in the spirit of the local index theorem of Atiyah, Bott, Gilkey and Patodi \cite{ABP,Gilkey,Pat}. This local refinement was first stated and proved by Getzler \cite{G2}. We will sketch a different proof, elaborated in \cite{K1}. Before we can state the theorem, however, we need to discuss some of the theory of Dirac operators coupled to superconnections.

Dirac operators coupled to superconnections on compact manifolds are elliptic, and one may develop their theory in a manner completely analogous to that of Dirac operators coupled to ordinary connections. We will see, however, that the $\R^\times$-action introduced in Eq.~\ref{eq:raction} will enter crucially in the local index theorem: the constants in the various elliptic estimates pick up a dependence on the parameter of the action on superconnections, and it is importance to keep track of this dependence. The theory is carefully developed in \cite{K1}, and it is discovered that the constants in the important estimates become (non-constant) functions of the parameter $s$ in Eq.~\ref{eq:raction} which may be chosen to be rational functions in $|s|^{1/2}$. We will quote the required estimates as needed, referring the interested reader to \cite{K1} for proofs. As with Dirac operators coupled to ordinary connections, one discovers that when the superconnection is unitary, the Laplacian $\Dirac(\Sc)^2:\Ll(X;\Sp(X)\stimes V)\to \Ll(X;\Sp(X)\stimes V)$ is self adjoint, and may thus be diagonalised. Its spectrum is real, non-negative and discrete, with finite dimensional eigenspaces. Furthermore, the eigensections are smooth. The diagonalisation respects the $\Zz$-grading, so that each eigenspace is itself $\Zz$-graded, and the Dirac operator restricted to that eigenspace is an odd endomorphism. For non-zero eigenvalues it is an isomorphism.

Associated to the operator $\Dirac(\Sc)$ one has the heat semigroup:
$e^{-t\Dirac(\Sc)^2}.$
A priori, for each $t>0$, this is a bounded linear operator from $\Ll(X;\Sp(X)\stimes V)$ to $\Ll(X;\Sp(X)\stimes V)$, and trace class. It is in fact considerably better -- it is smoothing for each $t>0$, and smooth in $t$. It is thus represented by an integral kernel $p_t(x,y)$, $t\in\R^{>0}$, $x$, $y\in X$, smooth in all of its variables, and symmetric in $x$ and $y$. We call this the heat kernel associated to $\Sc$. If we consider the family $\Sc^s$ (defined in Eq.\ref{eq:raction}), we get a smooth family of heat kernels $p_{t,s}(x,y)$. The McKean-Singer formula \cite{MS} shows
\begin{equation}\label{eq:mckeansinger}
\Index \dslash(\Sc)=\int_X\ud x\;\sTr p_{t,s}(x,x),
\end{equation}
for any $t\in\R^{>0}$, $s\in\R^\times$. We are now in a position to state the local index theorem of Getzler \cite{G2}.
\begin{theorem}\label{th:localindex} Let $X$ be a compact, spin and Riemannian manifold, with finite dimensional complex and hermitian $\Zz$-graded vector bundle $V\to X$ with unitary superconnection $\Sc$. Then
\begin{equation}\label{eq:index}
\lim_{t\to 0}\sTr p_{t,1/t}(x,x)\;\ud x=(2\pi i)^{-\dim X/2}\left[\Ahat(\Curv{X})\sCh\Sc\right]_\bb{n},
\end{equation}
where $p_{t,s}(x,y)$ is the integral kernel associated with the heat semigroup $\exp[-t\Dirac(\Sc^s)^2]$.
\end{theorem}
Note the the appearance of the $\R^\times$-action (Eq.~\ref{eq:raction}) in the limit (Eq.~\ref{eq:index}). We also note that Th.~\ref{th:localindex} along with the McKean-Singer formula implies the index theorem (Th.~\ref{eq:indextheorem}) for Dirac operators coupled to superconnections.

Getzler proves Th.~\ref{th:localindex} using stochastic arguments and the Feynman-Kac formula. We sketch a proof relying on heat kernel methods and the Getzler scaling argument. We will see that an essential pre-requisite for the argument is the existence of a small $t$ asymptotic expansion for $p_{t, 1/t}(x,x)$, and it is to this that we turn our attention.
\subsection{A small $t$ asymptotic expansion of the heat kernel}
We wish to establish the following proposition.
\begin{proposition}\label{prop:asymptote}
The heat kernel has a small $t$ asymptotic expansion on the diagonal given by
\[
p_{t,s/t}(x,x)\sim \sum_{k\ge 0}A_k(x,s)t^{k-\dim X/2},
\]
the $A_k(x,s)\in \Cliff(X)_x\stimes\End(V_x)$ are local, and vary smoothly in $x$ and $s$.
\end{proposition}
Notice that this is different from the usual small $t$ asymptotic expansion of the heat kernel in that it couples the parameter in the $\R^\times$-action on superconnections (Eq.~\ref{eq:raction}) with time. The usual methods of proving the existence of the small $t$ asymptotic expansion seem insufficient to address the situation, and we are forced to use other techniques.

We approach the small $t$ behaviour of the heat kernel rather indirectly, using parabolic scaling to exchange the small $t$ behaviour for the behaviour of the heat kernel under blowup. To this end, we now focus our attention near a point $x\in X$. We will use the exponential map to identify a neighbourhood of $x$ with a ball in $T_xX$, and transfer the geometry to there. We will extend the geometry to all of $T_xX$ in a controlled manner, and consider the behaviour of the heat kernel for this geometry under the scaling $x\mapsto x/\eps$, and show that it varies smoothly around $\eps=0$. Finally we will use parabolic scaling to relate the small $t$ behaviour of the heat kernel on $X$ near $x$ to the small $\eps$ behaviour of the heat kernel on $T_xX$.

Let $r>0$ be such that 
\[
\exp:B_r(0)\to U
\]
be a diffeomorphism, where $B_r(x)\subset T_x(X)$, and $U\subset X$. Choose a monotone function $\rho\in C^\infty(\R)$ such that $\rho(x)=0$ for $x>r$, and $\rho(x)=1$ for $x<3r/4$. Pull back the geometry on $X$ to $B_r(0)$ with the exponential map and obtain a complex, hermitian $\Zz$-graded vector bundle $E\to B_r(0)$, with superconnection $\Sc^s$, and a metric $g$ and bundle of spinors $\Sp\to B_r(0)$ (we abuse notation slightly, suppressing the $\exp^*$). We trivialise $E\to B_r(0)$ and $\Sp\to B_r(0)$ by parallel transport along radial geodesics. We extend $E$ and $S$ trivially to all of $T_x(X)$, and extend the superconnection and metric using $\rho$, such that on the complement of $B_r(0)$ the geometry is constant (we set it equal to that at the origin). Precisely, define 
\[
 g=\rho\exp^*g+(1-\rho)g(0),
\]
and $\Sc^s$ similarly. We now have put a metric on $T_x(X)$, and endowed it with trivial, complex and hermitian $\Zz$-graded vector bundles $E$, $\Sp\to T_x(X)$, with $E$ endowed with the superconnection $\Sc^s$. 

We define the operators $T_\eps:T_x(X)\to T_x(X)$ by $T_\eps(\xi)=\eps \xi$. With the help of these operators, we define the family of heat operators
\begin{equation}\label{eq:heps}
 H_\eps=\frac{\partial}{\partial t}+\eps^2T^*_\eps \Dirac\left(\Sc^{s/\eps^2}\right)^2\left(T^*_\eps\right)^{-1}
\end{equation}
with heat kernels $p^\eps_{t,s}(\xi,\xi')$. One may show using Fourier analysis (see \cite{K1}, Ch.~3, Sec 3) that the heat kernels $p^\eps_{t,s}(\xi,\xi')$ vary smoothly in all of their variables (and crucially in $\eps$) for $t$ bounded away from zero.
 
The next lemma is key, relating the heat kernel on $X$ near $x$, and the scaled heat kernel on $T_xX$.
\begin{proposition}\label{lem:kerneleps}
For any $r'<3r/4$, there exist a constant $c$, and a function $C(u)\in\R[|u|^{-1/2},|u|^{1/2}]$ such that for $|\xi|<r'$, and $|\eps|$ small enough,
\[
\left|\exp^*p_{t,s}(0,\xi)-|\eps|^{-n}p^\eps_{t/\eps^2,\eps^2s}(0,\xi/\eps)\right|\le C\left(\eps^2s\right) e^{-c/t}.
\]
In fact, $c$ may be chosen to be ${r'}^2/4$.
\end{proposition}
\begin{proof}
We fix $x\in X$. Let $\delta=(3r/4-r')/4$ and $\rho\in C^\infty(X)$ be a monotone function such that $\rho(y)=1$ for $\dist(x,y)<r'$, and $\rho(y)=0$ for $\dist(x,y)>3r/4-\delta$. We define $\phi_t$ on $X$ as follows
\[
 \phi_t(y)=\begin{cases}
p_{t,s}(x,y)-\rho(y)\eps^{-n}p^\eps_{t/\eps^2,\eps^2s}(0,\exp^{-1}(y)/\eps)&y\in U,\\
p_{t,s}(x,y)&y\in X\setminus U.
\end{cases}
\]
We wish to estimate the $C^k$-norm of $H\phi_t$, where
\[
 H=\frac{\partial}{\partial t}+\Dirac\left(\Sc^s\right)^2
\]
is the heat operator associated with $\Sc^s$. By definition 
\[
 Hp_{t,s}(x,y)=0,
\]
and a short computation shows
\[
 Hp^\eps_{t/\eps^2,\eps^2s}(0,\exp^{-1}(y)/\eps)=0,
\]
for $y\in U$. Thus the only contribution to $H\phi_t$ comes from
\[
H\left(\rho(y)\eps^{-n}p^\eps_{t/\eps^2,\eps^2s}(0,\exp^{-1}(y)/\eps)\right)
\]
for $y$ such that $r'<\dist(x,y)<3r/4-\delta$. A little algebra shows that  
\begin{multline*}
H\phi_t(y)=-\eps^{-n}\bigg[c(\nabla^2\rho)p^\eps_{t/\eps^2,\eps^2 s}(0,\exp^{-1}(y)/\eps)-2\left\langle\ud\rho,\nabla p^\eps_{t/\eps^2,\eps^2 s}(0,\exp^{-1}(y)/\eps)\right\rangle\\
+\sum_{i\ge2}[c(\omega^s_i),c(\ud \rho)]p^\eps_{t/\eps^2,\eps^2 s}(0,\exp^{-1}(y)/\eps)\bigg],
\end{multline*}
where we have used the decomposition $\Sc^s=\nabla+\omega^s$.

The $C^k$ norms of $\rho$, $\omega^s$ are bounded, and by the decay of the heat kernel off the diagonal\footnote{This is Prop.~3.24 in \cite{K1}: For any small enough $d>0$, there exists a constant $c>0$ (independent of the superconnection parameter) and $C(s)\in\R[|s|^{-1/2},|s|^{1/2}]$ such that
\[
\|p_{t,s}(x,y)\|_{C^k}\le C(s)e^{-c/t}
\]
for $\dist(x,y)>d$ and $t>0$. The constant $c$ may be chosen less than $d^2/4$; $C(s)$ then depends only on the superconnection, $s$ and $k$.
}
\[
 \left\|p^\eps_{t/\eps^2,s\eps^2}(0,\exp^{-1}(y)/\eps)\right\|_{C^k}\le C(s\eps^2)e^{-{r'}^2/4t},
\]
where $C(u)\in\R[|u|^{1/2},|u|^{-1/2}]$.

Combining these estimates shows
\[
\left\|H\phi_t\right\|_{C^k}\le C(s\eps^2) |\eps|^{-n}e^{-{r'}^2/4t},
\]
for some $C(u)\in\R[|u|^{1/2},|u|^{-1/2}]$.

We now wish to show $\lim_{t\to 0}\phi_t=0$. Both $p_{t,s}(x,y)$ and $\eps^{-n}p^\eps_{t/\eps^2,s\eps^2}(0,\exp^{-1}(y)/\eps)$ converge to $\delta_x\id_{E_x}$ in $H^l$, $l<-\dim X/2$ as $t\to0$, so that 
\[
 \lim_{t\to 0}\phi_t=0
\]
in $H^l$, $l<-\dim X/2$. We define
\[
 \psi_t=-\int_0^t\ud u\,e^{-(u-t)\Dirac(\Sc^s)^2}(H\phi)_u,
\]
and note that it is smooth, and converges to zero in $C^\infty$ as $t\to 0$. But $H(\phi_t+\psi_t)=0$, and thus, by uniqueness of solutions to the heat equation, and $\lim_{t\to 0}(\phi_t+\psi_t)=0$ in $H^l$, $l<-\dim X/2$, we conclude
\[
 \psi_t=-\phi_t.
\]
But then $\phi_t$ converges to zero as $t\to0$ in $C^\infty$, as asserted.

Finally, the basic parabolic estimate for the heat semigroup\footnote{This is Prop.~3.23 in \cite{K1}:
 Let $H_s=\frac{\partial}{\partial t}+\Dirac(\Sc^s)^2$ be the heat operator associated to $\Sc^s$. There exist $C_l\in\R[|s|^{-1/2},|s|^{1/2}]$ such that for any $\phi(x,t)\in C^\infty(X\times[0,\infty); \Sp\stimes E)$ and all $t\ge0$
\[
 \|\phi(\cdot,t)\|_{H^l}\le C_l(s)\left(\int_0^t\ud u\,\|(H_s\phi)(\cdot,u)\|_{H^l}+\|\phi(\cdot,0)\|_{H^l}\right).
\]}  
 shows
\[
 \|\phi_t\|_{H^l}\le C(s\eps^2) |\eps|^{-n}e^{-{r'}^2/4t},
\]
for some $C(u)\in\R[|u|^{-1/2},u^{1/2}]$, where the degree of $C(u)$ depends on $l$.
Choosing $l$ large enough and using the Sobolev embedding theorem we obtain the result.
\end{proof}

We may now establish Prop.~\ref{prop:asymptote}. We begin with the small $\eps$ asymptotic expansion
\[
 p^\eps_{1,s}(0,0)=\sum_{k\ge0}^Nb_k\eps^k+O\left(|\eps|^{N+1}\right).
\]
 The $b_k$ are smooth in $s$ by virtue of the fact that $p^\eps_{t,s}(\xi,\xi')$ varies smoothly in $s$. The $b^i$ are local in the geometry of $X$: indeed, all $\eps$ derivatives of $H^\eps$ at $\eps=0$ have coefficients depending only on the jet of the geometry at $x$. Prop.~\ref{lem:kerneleps} applied twice shows that
\[
 \left|p^\eps_{1,s}(0,0)- p^{-\eps}_{1,s}(0,0)\right|<C(s\eps^2)e^{-c/\eps^2},
\]
allowing us to conclude (recalling that $C(s\eps^2)$ may be taken to be a rational function in $|s\eps^2|^{1/2}$) that up to an error that vanishes to all orders of $\eps$, $p^\eps_{1,u}(0,0)$ is even in $\eps$, and the coefficients $b_{2l+1}$ all vanish. Prop.~\ref{lem:kerneleps} also shows that
\[
\left|\exp^* p_{t,s/t}(0,0)-\eps^{-n}p^\eps_{t/\eps^2,s\eps^2/t}(0,0)\right|\le C(s\eps^2/t)e^{-c/t},
\]
so that, setting $\eps=\sqrt{t}$, one has
\[
 \exp^* p_{t,s/t}(0,0)=t^{-n/2}p^{\sqrt{t}}_{1,s}(0,0)+O\left(C(s)e^{-c/t}\right),
\]
or, setting 
\[
 A_k=b_{2k},
\]
one arrives at
\[
 \exp^* p_{t,s/t}(0,0)=t^{-n/2}\sum_{k\ge0}^NA_kt^k+O\left(t^{N-n/2+1/2-\delta }\right),
\]
where $\delta$ is arbitrarily small but non-zero.
This establishes the existence of the desired asymptotic expansion.
\subsection{Proving the local index theorem}
We are now in a position to sketch a proof of Th.~\ref{th:localindex}. We will follow the general strategy found in the Freed's online lecture notes on Index Theory, and use Getzler's scaling \cite{G} to find the constant term in the small $t$ asymptotic expansion $\sTr p_{t,1/t}(x,x)$, and show that all divergent terms vanish.

To begin we introduce an algebraic deformation of the Clifford algebra.  Let $\Cliff^\eps(X)$ be the bundle of Clifford algebras with generic fibre at $x\in X$ given by $\Cliff(T_xX,\eps^2 g)$. A $\Cliff(X)$ bundle is a $\Cliff^\eps(X)$ bundle in the obvious way; denote by $c_\eps(\cdot)$ the action of $\Cliff^\eps(X)$. 

There is a canonical vector bundle isomorphism $\Cliff(X)\cong\Forms^\bullet(X)$. We denote the image of $a\in\Forms(X)$ in $\Cliff(X)$ under this isomorphism by $\hat{a}$. Let $U_\eps:\Forms^\bullet(X)\to\Forms^\bullet(X)$ be given by
\begin{equation}\label{eq:algscale}
U_\eps:a\mapsto\eps^{-k}a,
\end{equation}
where $a\in\Forms^k(X)$, and $\hat{U}_\eps:\Forms^\bullet(X)\to\Cliff(X)$ be the induced map via the isomorphism $\Forms^\bullet(X)\cong\Cliff(X)$. 
\begin{lemma}
Let $a\in\Forms^k(X)$. Then
\[
\lim_{\eps\to0}\eps^l c_\eps(\hat{U}_\eps(a))=\begin{cases}
0&\text{if $l>k$},\\
\extm(a)&\text{if $k=l$},\\
\infty&\text{if $l<k$}.
\end{cases}
\]
\end{lemma}
An immediate consequence of this is the following.
\begin{lemma}\label{lem:sclim}
Writing $\Sc^s=\nabla+\omega_s$, one has
\[
\lim_{\eps\to0}\eps c_\eps(\hat{U}_\eps\omega_{s/\eps^2})=\extm(\omega_s).
\]
\end{lemma}

We return to the situation in the previous section: we concentrate on a neighbourhood $U$ of a point $x\in X$, identified via the exponential map with a ball $B_r(0)$ in $T_xX$, and use parallel transport along radial geodesics to trivialise the tangent bundle, the bundle of spinors and $V$ on $U$. We form the family of heat operators
\begin{eqnarray*}
\tilde{H}_\eps&=&\frac{\partial}{\partial t}+P_\eps,\label{eq:heatfamily}\\
P_\eps&=&\eps^2 S_\eps\Dirac(\Sc^{s/\eps^2})S^{-1}_\eps,
\end{eqnarray*} 
where $S_\eps=U_\eps T^*_\eps$.
The following lemma is crucial.
\begin{lemma}\label{lem:heatlim}
Let $\set{\xi_k}$ be an orthogonal linear co-ordinate system on $T_xX$. Then 
\begin{equation}
\lim_{\eps\to0}P_\eps=-\sum_{k,l}\left(\frac{\partial}{\partial \xi^k}-\frac{1}{4}\extm\left(\Omega^X_{kl}(x)\right)\xi^l\right)^2+\extm(\curv{\Sc^s})(x).
\end{equation}
Here $\Omega^X_{kl}$ is the Riemannian curvature.
\end{lemma}
The lemma is proved similarly to the analogous lemma for connections, taking into account lemma \ref{lem:sclim}. Let
\[
\tilde{H}_0=\lim_{\eps\to0}\tilde{H}_\eps.
\]
By lemma \ref{lem:heatlim}, and Mehler's formula, the heat kernel for $\tilde{H}_0$ is given by
\begin{equation}\label{eq:limitkernel}
 \tilde{p}_{t,s}(0,0)=(4\pi t)^{-n/2}\sqrt{\det\left(\frac{t\Omega^X/2}{\sinh(t\Omega^X/2)}\right)}e^{-t(\Sc^s)^2},
\end{equation}
where all curvatures are evaluated at the origin.

We are now in a position to finish the proof of Th.~\ref{th:localindex}. We define 
\begin{equation}\label{eq:pgetzler}
 \tilde{p}^\eps_{t,s}(x,y)=\eps^nS_\eps p_{\eps^2 t,s/\eps^2}(x,y).
\end{equation}
A short computation shows that $\widetilde{H}_\eps\tilde{p}^\eps_{t,s}(x,y)=0$, with initial condition $\delta_x$, so that, by uniqueness of solutions, $\tilde{p}^\eps_{t,s}(x,y)$ is the heat kernel for $\widetilde{H}_\eps$. Fixing an orthonormal basis $\{e^i\}$ of $T^*_x(X)$, we write
\[
 p_{t,s}(x,y)=\sum_I p^I_{t,s}(x,y)e^I,
\]
where the $I$ are ordered multi-indices, the $e^I\in\Cliff(X)_x$, and $p^I_{s,t}(x,y)\in\End(E_x)$. Then
\[
 \tilde{p}^\eps_{t,s}(x,y)=\sum_I \eps^{n-|I|}p^I_{\eps^2t,s/\eps^2}(x,\eps y)\hat{e}^I.
\]
The small $t$ asymptotic expansion of $p_{t,s/t}(x,x)$ may be written
\[
 p_{t,s/t}(x,x)\sim(4\pi)^{-n/2}\sum_{j,I}t^je^IA_{j,I},
\]
so that
\[
 \tilde{p}^\eps_{t,s/t}(x,x)\sim(4\pi)^{-n/2}\sum_{j,I}\eps^{2j-|I|}t^{j-n/2}\hat{e}^IA_{j,I}.
\]
It is well known (see \cite{BGV} Ch.~3) that
\[
 \sTr\,\hat{e}^I=\begin{cases}
                  (-2i)^{n/2}&\text{if $I\sim12\ldots n$,}\\
		  0&\text{otherwise,}
                 \end{cases}
\]
and we compute
\[
 \sTr\,\tilde{p}^\eps_{t,s/t}(x,x)\sim(4\pi)^{-n/2}\sum_{j}\eps^{2j-n}t^{j-n/2}\sTr\,A_{j,12\ldots n}.
\]
Eq.~\ref{eq:limitkernel} shows that
\[
 \tilde{p}^0_{t,s/t}(x,x)=(4\pi)^{n/2}\sum_jP_j(\Omega^X/2,-(\Sc^{s/t})^2)t^{j-n/2}.
\]
The coefficient of $t^{j-n/2}$ is a $2j$-form. On the other hand, the coefficient of $t^{j-n/2}$ is given by
\[
 \sum_{I}\eps^{2j-|I|}\hat{e}^IA_{j,I}.
\]
Thus $A_{j,I}=0$ for $2j<|I|$. We see then that 
\[
\lim_{t\to0}\sTr p_{t,s/t}(x,x)=\sTr A_{n/2,12\ldots n}.
\] 
However
$A_{n/2,12\ldots n}$ is the constant term in the small $t$ asymptotic expansion of $\tilde{p}^0_{t,s/t}(x,x)$. Examining this asymptotic expansion, and noting that the $k$-form part of $t(\Sc^{s/t})^2$ is order $t^{k/2}$, we see that
\[
\sTr A_{n/2,12\ldots n}=\lim_{t\to0}\sTr\tilde{p}_{t,s/t}(x,x)=(2\pi i)^{n/2}\left[\Ahat\left(\Omega^X\right)\sCh\left(\Sc^s\right)\right]_\bb{n},
\]
proving the result.
\section{$\eta$-invariants and manifolds with boundaries}\label{sec:boundary}
We now turn our attention to manifolds with boundary. We fix a compact Riemannian and spin manifold $X$ with boundary $\partial X=Y$. Let $V\to X$ be a finite dimensional, complex and hermitian $\Zz$-graded vector bundle with superconnection $\Sc$. We assume that a collar neighbourhood $U$ of the boundary exists and a homeomorphism $U\cong Y\times[0,1)$ such that the geometry respects the product structure, $V$ is pulled back from $V|_{\partial X}$, and $\Sc=\ud t\,\partial_t+\Sc|_{\partial X}$, where $t$ is the parameter on $[0,1)$. We shall define an $\eta$-invariant on $Y$ that obeys an APS theorem \cite{APS1} when $Y$ is a boundary as described.

We begin by defining the Chern-Simons form for a pair of superconnections.
\begin{definition}\label{def:chern-simons}
 Let $Y$ be a smooth manifold, and $V\to Y$ be a finite dimensional, $\Zz$-graded vector bundle, with superconnections $\Sc_0$, $\Sc_1$. Fix once and for all $\rho\in C^\infty(\R)$ monotone such that  
$\rho(t)=0$ (resp. $\rho(t)=1$) for $t<1/4$ (resp. $t>3/4$). The \emph{ Chern-Simons form}  associated with the pair of superconnections is given by
\[
 \alpha[\Sc_0,\Sc_1]=\int_0^1\,\ud t\,\sTr\left[\dot{\rho}(t)\left(\Sc_1-\Sc_0\right)e^{((1-\rho(t))\Sc_0+\rho(t)\Sc_1)^2}\right].
\]
\end{definition}
Having done so, we may now define the $\eta$-invariant (mod $\Z$) for a superconnection.
\begin{definition}\label{def:tau}
 Let $Y$ be a compact and spin Riemannian manifold, and $V\to Y$ be a finite dimensional, complex and hermitian $\Zz$-graded vector bundle with unitary superconnection $\Sc$. The \emph{ $\tau$-invariant} associated to this data is given by
\[
 \tau_Y(\Sc)=\tau_Y(\nabla)\exp\left[(2\pi i)^{(1-\dim Y)/2}\int_Y\,\Ahat\left(\Omega^Y\right)\alpha[\Sc,\nabla]\right],
\]
where we have written $\Sc=\nabla+\omega$, and $\nabla$ is a connection on $V$.\footnote{We remind the reader that associated to an odd, spin, riemannian manifold $Z$ with hermitian vector bundle $E\to Z$ and connection $\nabla$ is the $\eta$-invariant $\eta_Z(\nabla)$, defined by Atiyah, Patodi and Singer \cite{APS1} to be 
\[
\eta_Z(\nabla)=\lim_{s\to 1}\sum_{\lambda\in\spec{\dslash(\nabla)}}\frac{|\lambda|}{\lambda^{1+s}},
\]
where here the limit is to be interpreted as a renormalised limit. Associated to this invariant are two secondary invariants: the reduced $\eta$-invariant
\[
\xi_Z(\nabla)=\frac{\eta_Z(\nabla)+\dim\ker\dslash(\nabla)}{2},
\]
and the exponentiated $\eta$-invariant
\[
\tau_Z(\nabla)=e^{2\pi i\xi_Z(\nabla)}.
\]
}
\end{definition}
The reader may well ask about the well-definedness of the putative $\eta$-invariant. The only choice is in the decomposition $\Sc=\nabla+\omega$. This choice can be made completely canonical, by insisting that $\omega$ have no one-form term. However, we prefer to note that an easy consequence of Th.~\ref{thm:APS} is that any choice gives rise to the same invariant.

The reader may also wish some justification for the definition. The essential justification is that we desire an $\eta$-invariant that obeys an APS theorem. Following the usual definition of the $\eta$-invariant as a spectral invariant unfortunately does not give rise to a geometric invariant that satisfies an APS theorem with the right hand side expressible in terms of recognisable characteristic classes (essentially as Th.~\ref{th:localindex} involves the entire family of operators $\Dirac(\Sc^s)^2$, which cannot be simultaneously diagonalised), and we were thus forced to look elsewhere. Our definition comes from considering the cylinder $Y\times[0,1]$, with the $\Zz$-graded vector bundle $V\times[0,1]\to Y\times[0,1]$, and the superconnection on the extended bundle such that on $V\times\{0\}$ it restricts to $\Sc$, and on $V\times\{1\}$ it restricts to $\nabla$. Applying the (putative) APS theorem to this situation gives rise to our definition.

It is almost a tautology that the APS-theorem holds.
\begin{theorem}\label{thm:APS}
 Let $X$ be a compact and spin Riemannian manifold with boundary $\partial X=Y$. Let $E\to X$ be a finite dimensional, complex, hermitian, $\Zz$-graded vector bundle, and $\Sc$ a unitary superconnection on $E$. Suppose there exists an open neighbourhood $Y\subset U\subset X$ homeomorphic to $Y\times [0,1)$ such that the geometry on $U$ is pulled back from that on $Y$ extended trivially. Then
\[
 \tau_Y\left(\Sc|_Y\right)=\exp\left[-(2\pi i)^{1-\dim X/2}\int_X\,\Ahat\left(\Omega^X\right)\sCh(\Sc)\right].
\]
\end{theorem}
\begin{proof}
Write $\Sc=\nabla+\omega$. Let $\bar{E}\to\bar{X}$, with $\bar{Y}=\partial\bar{X}$ denote the data with the opposite choice of orientation. Let $\widetilde{Y}=Y\times I$, $\widetilde{E}=E|_Y\times I$, and let $\widetilde{\Sc}=(1-\rho(t))\Sc|_Y+\rho(t)\nabla|_Y$, where $\rho$ is the smooth function in Def.~\ref{def:chern-simons}.
Form 
\[
\mathcal{X}=X\cup_Y\widetilde{X}\cup_Y \bar{X},
\]
with $\mathcal{E}\to\mathcal{X}$ defined by
\[
 \mathcal{E}=E\cup_{E|_Y}\widetilde{E}\cup_{E|_Y}\bar{E},
\]
with superconnection $\Sc^{\mathcal{E}}$ glued from $\Sc$ on $E$, $\widetilde{\Sc}$ on $\widetilde{E}$, and $\nabla$ on $\bar{E}$ (see Fig. \ref{fig:eta}). The index theorem applied to this situation shows
\[
 (2\pi i)^{-n/2}\int_{\mathcal{X}}\,\Ahat\left(\Omega^{\mathcal{X}}\right)\sCh\left(\Sc^{\mathcal{E}}\right)=0\mod\Z.
\]
But
\begin{multline*}
 \int_{\mathcal{X}}\,\Ahat\left(\Omega^{\mathcal{X}}\right)\sCh\,\Sc^{\mathcal{E}}=\\\int_{X}\,\Ahat\left(\Omega^{X}\right)\sCh\,\Sc+\int_Y\,\Ahat\left(\Omega^Y\right)\alpha[\Sc|_Y,\nabla]+\int_{\bar{X}}\,\Ahat\left(\Omega^{\bar{X}}\right)\sCh\,\nabla.
\end{multline*}
Applying the APS theorem for $\xi$-invariants of connections, we have
\[
 0=(2\pi i)^{-n/2}\int_{X}\,\Ahat\left(\Omega^{X}\right)\sCh\,\Sc+(2\pi i)^{-n/2}\int_Y\,\Ahat\left(\Omega^Y\right)\alpha[\Sc|_Y,\nabla]-\xi_{\bar{X}}(\nabla) 
\]
modulo integers.
Exponentiating both sides, and noting $\xi_{\bar{X}}(\nabla)=-\xi_{X}(\nabla)$, one obtains the theorem.
\begin{figure}[hb]
\centering

             \includegraphics[width=0.8\linewidth]{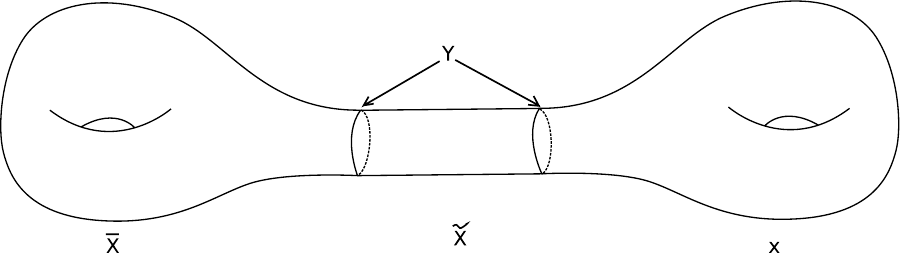}
           
 \caption{The manifold $\mathcal{X}$.}
 \label{fig:eta}
\end{figure}
\end{proof}
\section{An index theorem for families}\label{sec:families}
\subsection{Riemannian maps}
The good notion of ``family'' for Riemannian geometry is that of a Riemannian map.
\begin{definition}[Riemannian maps]\label{def:riemannmap}
 Let $X$, $Y$ be smooth manifolds. A \emph{ Riemannian map} is a triple $(\pi,\,g^{X/Y}\!,\,P)$, where $\pi:X\to Y$ is a smooth submersion, $g^{X/Y}$ be a positive, non-degenerate, symmetric bilinear form on $T(X/Y)$, and $P:T(X)\to T(X/Y)$ a smooth projection.
\end{definition}
\begin{remark}
Several remarks are in order. Generally, one simply denotes the Riemannian map by $\pi:X\to Y$, suppressing the metric and projection. Specifying the projection $P$ is equivalent to choosing a ``horizontal'' distribution in $T(X)$, and determines a splitting 
\[
 T(X)\cong T(X/Y)\oplus\pi^*T(Y),
\]
and in particular, allows one to take a horizontal lift of vector fields $\xi\in T(Y)$, which we denote by $\tilde{\xi}$. A choice of metric $g^Y$ on $T(Y)$ makes $X$ a Riemannian manifold, with metric 
\begin{equation}\label{eq:riemannstructure}
g^X=g^{X/Y}\oplus\pi^*g^Y. 
\end{equation}
 Conversely, a Riemannian map $\pi:X\to Y$ along with a Riemannian structure on $X$ such that the metric splits as in Eq.~\ref{eq:riemannstructure} is called a Riemannian \emph{submersion}.
A Riemannian map $X\to\{\text{pt}\}$ is a Riemannian structure on $X$. 
\end{remark}
\begin{proposition}[Levi-Civita connection]
 Let $\pi:X\to Y$ be a Riemannian map. There exists a unique torsion free connection $\nabla^{\pi}$ on $T(X/Y)$ preserving $g^{X/Y}$. It is called the Levi-Civita connection of the map.
\end{proposition}
\begin{proof}[Construction]
 Choose a Riemannian structure on $Y$ and give $X$ a Riemannian structure by setting $g^X=g^{X/Y}\oplus\pi^*g^Y$. The Levi-Civita theorem asserts that there exists a unique unitary torsionfree connection $\nabla^X$ on $X$ preserving $g^X$. Form 
\[
 \nabla^{\pi}\equiv P\nabla^XP.
\]
One may check that this is independent of choices\footnote{Details of all constructions in this section may be found in \cite{BGV} Ch.~10.}.
\end{proof}
There are two important curvatures associated with a Riemannian map $\pi:X\to Y$ in addition to the curvature $\Omega^\pi$ of $\nabla^{\pi}$.
\begin{definition}
 We define \emph{ two curvature tensors} associated to a Riemannian map $\pi:X\to Y$. The first is the \emph{ second fundamental form}, $\secfun^\pi:T(X/Y)\stimes T(X/Y)\stimes\ker P\to\R$, and the second is \emph{ $T^\pi:T(Y)\stimes T(Y) \to T(X/Y)$} depending only on the Riemannian map. The tensor $\secfun^\pi$ is symmetric in the first two factors; $T^\pi$ is totally skew, and does not depend on $g^{X/Y}$. 
 
To define $\secfun^\pi$, we choose a Riemannian structure $g^Y$ on $Y$, and induce a Riemannian structure on $X$. Define
\[
 \secfun^\pi:\alpha\stimes\beta\stimes\xi\mapsto-g^{X/Y}(P\nabla^X_\alpha\xi,\beta),
\]
where $\alpha$, $\beta\in T(X/Y)$, $\xi\in\ker P$. One may show that this is independent of the choice of Riemannian structure on $Y$.

We define $T^\pi$ as follows. Let $\xi$, $\chi\in T(Y)$. Then
\[
 T^\pi:\xi,\,\chi\mapsto P([\tilde{\xi}, \tilde{\chi}]).
\]
\end{definition}
We may define the mean curvature of the family from the second fundamental form.
\begin{definition}
 The \emph{ mean curvature} $H^\pi\in\Omega^1(X)$ is a horizontal one-form defined as follows. Let $\xi\in T(X)$, and decompose $\xi$ as a sum of horizontal and vertical vectors:
\[
\xi=P\xi+\zeta.
\]
Then
\[
 H^\pi(\xi)=\sTr\secfun^\pi(\cdot,\cdot,\zeta).
\]
In terms of an orthonormal frame $\{e_i\}$ of $T(X/Y)$, one has
\[
 H^\pi(\xi)=\sTr\secfun^\pi(\cdot,\cdot,\zeta)=\sum_i\secfun^\pi(e_i,e_i,\zeta).
\]
\end{definition}

There is a natural $\R^{\times}$-action associated with a Riemannian map $\pi:X\to Y$ which scales the vertical metric. It is given by
\begin{equation}\label{eq:riemannaction}
 \pi^t=(\pi,\,|t|^{-1}g^{X/Y}\!,\,P).
\end{equation}
In order to understand the effect of the action, let us introduce a second Riemannian map $\rho:Y\to\{\mathrm{pt}\}$ (in other words, a Riemannian structure on $Y$). This gives $X$ a Riemannian structure, with metric $g^X=g^{X/Y}\oplus\pi^*g^Y$. The Levi-Civita connection associated to this metric is $\nabla^{\rho\circ\pi}$. It is independent of a global scaling of the metric, and thus
\begin{equation}\label{eq:rescaledconnection}
 \nabla^{\rho^t\circ\pi}=\nabla^{\rho\circ\pi^{1/t}}.
\end{equation}
The limit $t\to0$ of $\rho^t\circ\pi$ is called the \emph{adiabatic limit}. We will sometimes denote the adiabatic limit of a geometric quantity $\Phi$ by $\alim\Phi$. The following proposition summarises the geometry of the adiabatic limit.
\begin{proposition}
 The limit $\alim\nabla^X=\lim_{t\to0}\nabla^{\rho^t\circ\pi}$ exists, and is torsion free. It is incompatible with any Riemannian structure on $X$. Relative to the decomposition $T(X)\cong T(X/Y)\oplus T(B)$, the connection may be written
\[
 \alim\nabla^X=\begin{pmatrix}
  \nabla^\pi&*\\
&\nabla^\rho
 \end{pmatrix},
\]
and thus
\[
 \alim\Curv{X}=\begin{pmatrix}
                  \Curv{\pi}&*\\
			    &\Curv{\rho}
                 \end{pmatrix}.
\]
In particular, in the adiabatic limit the form $\Ahat(\Curv{X})$ exists, and decomposes as
\[
 \alim\Ahat(\Curv{X})=\pi^*\Ahat(\Curv{Y})\wedge\Ahat(\Curv{\pi}).
\]
\end{proposition}
\begin{remark}\label{rem:clifflim}
The Clifford bundle $\Cliff(X)$ also has an ``adiabatic limit",\footnote{This should be taken as the \emph{definition} of the ``adiabatic limit" of $\Cliff(X)$. One may in fact see $\alim\Cliff(X)$ as a limit of Clifford algebras in the same manner that, for a vector space $V$ with inner product $g$, the exterior algebra $\bigwedge V$ may be seen as the $\eps\to0$ limit of Clifford algebras $\Cliff(V,\eps g)$, but for our purposes we do not need to make this formal. }   the Clifford algebra bundle associated to $T^*X$ with the degenerate metric $\alim g^X$: by definition $\alim\Cliff(X)=\Cliff(T^*X,\,\alim g^X)$. There is an isomorphism of algebra bundles:
\[
 \alim\Cliff(X)=\pi^*\bigwedge T^*(Y)\stimes\Cliff(X/Y);
\]
morally, the horizontal multiplication degenerates to exterior multiplication. If the fibres of $\pi$ are spin, then the bundle $\Sp_0=\pi^*\bigwedge T^*(Y)\stimes \Sp(X/Y)$ is a left module for $\alim\Cliff(X)$, with the action being Clifford multiplication in vertical directions, and exterior multiplication in horizontal directions. We denote this action by $c_0(\cdot)$. The module $\Sp_0$ plays a similar role to spinors.

We recall that the Clifford algebra of a vector space with metric is naturally isomorphic as a vector space to the exterior algebra. Thus there is an isomorphism of vector bundles
\begin{equation}
\sigma:\Cliff(X/Y)\to\bigwedge T^* (X/Y),
\end{equation}
which we call the symbol map. We extend this to an isomorphism of vector bundles
\begin{equation}\label{eq:cliffsymbol}
\alim\Cliff(X)\cong\pi^*\bigwedge T^*(Y)\stimes\Cliff(X/Y)\to\bigwedge T^*(X).
\end{equation}
There is a natural bigrading on $\bigwedge T^*(X)$ given by
\begin{eqnarray}\label{eq:bigrade}
\wedge T^*(X)&=&\sum_{p,q}\wedge^{p,q}T^*X,\\
\wedge^{p,q}T^*X&=&\wedge^pT^*(Y)\stimes\wedge^q T^*(X).
\end{eqnarray}
We denote the $(p,q)$ component $\alpha\in\Omega(X)$ by $[\alpha]_\bb{p,q}$.
\end{remark}
\subsection{Pushforwards}
Let $\pi:X\to Y$ be a Riemannian map, with compact and spin fibres, and $E\to X$ a complex, hermitian, finite dimensional, $\Zz$-graded vector bundle with unitary superconnection $\Sc$. We will construct an infinite dimensional complex and hermitian $\Zz$-graded vector bundle $\pi_!E\to Y$ with superconnection $\pi_!\Sc$ that represents the pushforward of this data to $Y$. This will be done in a functorial manner, and it is to be done in such a way that, on choosing a metric on $Y$, $\Dirac_Y(\pi_!\Sc)=\Dirac_X(\Sc)$. This construction will extend the construction of Bismut \cite{Bis} to superconnections.

\begin{definition}\label{def:pushforwardbundle} Let $\pi:X\to Y$ be a Riemannian family with compact and spin fibres, and $E\to X$ be a complex, hermitian, finite dimensional $\Zz$-graded vector bundle with superconnection. Define $\pi_!E$ to be the infinite dimensional Fr\'echet bundle with typical fibre 
\[
\pi_!E|_y=\Gamma\left(i_y^*\left(\Sp(X/Y)\stimes E\right)\right),
\]
where $i_y:\{\mathrm{pt}\}\to Y$, $i:\{\mathrm{pt}\}\mapsto y\in Y$, and $\Sp(X/Y)$ is the bundle of $\Zz$-graded vertical spinors. We endow $\pi_!E$ with a hermitian structure by taking the metric to be the $\Ll$ pairing on sections. Finally, define the superconnection $\pi_!\Sc$ on $\pi_!E$ to be the unique unitary superconnection on $\pi_!E$ such that, for any choice of metric on $Y$, $\Dirac_Y(\pi_!\Sc)=\Dirac_X(\Sc)$. 
\end{definition}
\begin{remark} We construct the superconnection $\pi_!\Sc$ below. It is easy to see that any superconnection satisfying $\Dirac_Y(\pi_!\Sc)=\Dirac_X(\Sc)$ is unique.
\end{remark}
\begin{remark}
The bundle $\pi_!E$ has a different character depending on whether we take smooth or $\Ll$ sections: in the former case it is a smooth Fr\'echet bundle, in the latter case it is a continuous Hilbert space bundle. We take the former as we wish to define superconnections.
\end{remark}

We begin by the construction of $\pi_!\Sc$ in the simplest case. Let $\mathbf{1}\to X$  be the purely even trivial line bundle with (super)connection the de Rham differential $\ud$. 

The bundle $\pi_*\mathbf{1}$ possesses a natural unitary connection $\nabla^{\pi_!\mathbf{1}}$. In order to describe it, we begin by noting that the horizontal differential of the relative volume form is just given by
\[
 \nabla^\pi_\xi\mathrm{vol}^{X/Y}=-H^\pi(\xi)\mathrm{vol}^{X/Y},
\]
where $\xi$ is a horizontal vector field. The appropriate connection on $\pi_!\mathbf{1}$ is thus given by
\[
 \nabla^{\pi_!\mathbf{1}}_\xi=\nabla^\pi_{\pi^*\xi}-\frac{1}{2}H^\pi(\pi^*\xi),
\]
for $\xi\in T(Y)$.

We must introduce one more notion before we are in a position to contruct pushed-forward superconnections. Let $E\to X$ be a finite dimensional vector bundle with superconnection $\Sc$. Then the fibrewise Dirac operator $\Dirac^\pi(\Sc)\in\Forms^0(Y;\End(\pi_!E)^\odd)$ is defined to be the Dirac operator of $\Sc$ restricted to the fibre: over a typical point $y\in Y$, $\Dirac^\pi(\Sc)_y=\Dirac(\Sc|_{\pi^{-1}y})$. 

With this understood, we may now state the following lemma, due to Bismut~\cite{Bis}.
\begin{lemma}\label{def:trivpush} Let $\pi:X\to Y$ be a Riemannian family with compact and spin fibres. Let $(\mathbf{1},\ud)\to X$ be the trivial $\Zz$-graded vector bundle with superconnection defined above. Then
\[
 \pi_!\ud=\Dirac^\pi(\ud)+\nabla^{\pi_!\mathbf{1}}+\frac{1}{4}c^\pi(T^\pi),
\]
where $c^\pi$ is vertical Clifford multiplication. 
\end{lemma}
\begin{remark}
Several remarks are in order. First, it is easy to compute that (see for example \cite{BGV} Ch.~10) 
\[
\Dirac_Y(\pi_!\ud)=\Dirac_X(\ud).
\]
Second, if $Y=\{\mathrm{pt}\}$, we have that $\pi_!\mathbf{1}\to\{\mathrm{pt}\}$ is simply the bundle of spinors on $X$, seen as a $\Zz$-graded vector space, and the pushforward $\pi_!\ud$ is exactly the standard Dirac operator on $X$. On the other extreme, if we take $X=Y$, so that fibres are single points, then $\pi_!\mathbf{1}\to Y$ is just the trivial even line bundle on $Y$, and $\pi_!\ud$ is the de Rham differential on the bundle.
\end{remark}
We now wish to construct the pushed-forward Dirac operator associated to a general $\Zz$-graded vector bundle over a Riemannian family. We need one more definition.
\begin{definition}
 Let $\pi:X\to Y$ be a Riemannian family, with compact and spin fibres. Let $\omega\in\Forms^i(X)$. Then the \emph{ pushforward} $\pi_!\omega\in\Forms(Y;\,\End(\pi_!\mathbf{1}))$ is the sum
\[
 \pi_!\omega=\sum_{j\le i}[\pi_!\omega]_\bb{j},
\]
where $[\pi_!\omega]_\bb{j}\in\Forms^j(Y;\,\End(\pi_!\mathbf{1}))$ is defined as
\[
 [\pi_!\omega]_\bb{j}(\tilde{\xi}^1,\ldots,\tilde{\xi}^j)=c^\pi\left(\intm(\tilde{\xi}^1)\cdots\intm(\tilde{\xi}^j)\omega\right),
\]
where the $\xi^k$ are vector fields on $Y$, $\intm$ denotes contraction, and $c^\pi(\cdot)$ denotes Clifford multiplication on $\Sp(X/Y)$ by the vertical part of the argument.
\end{definition}
With this understood, we may construct the pushforward for a general $\Zz$-graded vector bundle over a Riemannian family, with superconnection.
\begin{lemma}
 Let $\pi:X\to Y$ be a Riemannian family with compact and spin fibres. Let $E\to X$ be a finite dimensional, complex, hermitian, $\Zz$-graded vector bundle with superconnection $\Sc$.
 Decompose $\Sc=\nabla^E+\omega$, where $\nabla^E$ is an ordinary connection on $E$. Then
\[
 \pi_!\Sc=\Dirac^\pi(\nabla^E)+\nabla^{\pi_!E}+\frac{1}{4}c^\pi(T^\pi)+\pi_!\omega.
\]
Here
\[
 \nabla^{\pi_!E}=\nabla^{\pi_!\mathbf{1}}\stimes1+1\stimes\nabla^E.
\]
\end{lemma}
\begin{remark}
 If $\Sc$ is unitary, so is $\pi_!\Sc$.
\end{remark}
\begin{remark}
 We see that $[\pi_!\Sc]_\bb{0}=\Dirac^\pi(\Sc)$.
\end{remark}
\begin{remark}
 As required, if $Y$ is spin, then with any choice of Riemannian structure on $Y$, one easily sees that $\Dirac_Y(\pi_!\Sc)=\Dirac_X(\Sc)$. 
\end{remark}
\begin{remark}
 If $Y=\{\mathrm{pt}\}$, we have $\pi_!\Sc=\Dirac(\Sc)$, and if $X=\{\mathrm{pt}\}$, $\pi_!E=E$ and $\pi_!\Sc=\Sc$.
\end{remark}
\begin{remark}
 The $\R^{\times}$-action on Riemannian maps (Eq.~\ref{eq:rescaledconnection}), and the $\R^\times$-action on superconnections (Eq.~\ref{eq:raction}) intertwines:
\begin{equation}\label{eq:actionintertwine}
 \pi_!^t\Sc=(\pi_!\Sc^{1/t})^t.
\end{equation}
\end{remark}
\begin{remark}\label{rem:alimdirac}
 The superconnection may be very suggestively written in terms of the degenerate Clifford module $\Sp_0=\pi^*\bigwedge(T^*Y)\stimes \Sp(X/Y)$ (discussed in Rem. \ref{rem:clifflim}), which we recall is a module for $\alim\Cliff(X)$. In terms of the degenerate Clifford multiplication, the pushed-forward superconnection may be written
\[
 \pi_!\Sc=c_0(\alim \nabla^X\stimes1+1\stimes\Sc),
\]
where $\alim\nabla^X$ is the adiabatic limit of the Levi-Civita connection on $X$. In this way, $\pi_!\Sc$ may be regarded as the adiabatic limit of the Dirac operator coupled to $\Sc$ on $X$.
\end{remark}
\subsection{The Chern character}\label{sec:chernfam}
Let $\pi:X\to Y$ be a Riemannian family, with compact and spin fibres, and $E\to X$ be a finite dimensional, complex, hermitian, $\Zz$-graded vector bundle with unitary superconnection $\Sc$. The Chern character form associated with $\pi_!\Sc$ is defined by
\begin{equation}\label{eq:cherncharacterfamily}
 \sCh\pi_!\Sc=\sTr e^{-(\pi_!\Sc)^2}.
\end{equation}
Some work is needed to understand Eq.~\ref{eq:cherncharacterfamily}, and to show that the Chern character is well defined. Details may be found in the appendix of Ch.~9 in \cite{BGV}. We sketch the main points. A priori, we have $(\pi_!\Sc)^2\in\Forms(Y;\,\End(\pi_!E))^\even$. Calculation shows that at $x\in X$, $(\pi_!\Sc)^2$ is a local second order differential operator acting on $\Sp(X/Y)\stimes E$, with coefficients in $\bigwedge T^*(Y)_x$, of the form
\[
 (\pi_!\Sc)^2=\Dirac^\pi(\Sc)^2+N_x,
\]
where $N_x$ is a nilpotent first order differential operator, its coefficients being of strictly positive degree. Standard elliptic theory shows $e^{-\Dirac^\pi(\Sc)^2}$ to be trace class. Applying Duhamel's formula to $e^{-(\pi_!\Sc)^2}$ yields
\[
 e^{-(\pi_!\Sc)^2}=\sum_{i\ge0}(-1)^i\left[e^{-\Dirac^\pi(\Sc)^2}N_x\right]^{\# k}\#e^{-\Dirac^\pi(\Sc)^2},
\]
where
\begin{equation}\label{eq:Duhamel}
 e^AB\#e^C=\int_0^1\ud s\,e^{sA}Be^{(1-s)C}.
\end{equation}
The nilpotency of $N_x$ makes the sum finite, showing that $\sCh(\pi_!\Sc)\in\Omega(Y)$ is indeed well defined. 

Some remarks further remarks on the nature of the Chern character and the curvature are in order. We have seen that $\curv{\pi_!\Sc}$ is a second order elliptic differential operator on each fibre. The operator
\[
e^{-t(\pi_!\Sc)^2}
\]
is the semigroup associated to heat operator
\begin{equation}\label{eq:heatfibre}
\frac{\partial}{\partial t}+(\pi_!\Sc)^2
\end{equation}
acting on elements of $\wedge T^*Y_y\stimes(\pi_!E)_y$, $y\in Y$, where we recall that elements of $(\pi_!E)_y$ are sections of $(\Sp(X/Y)\stimes E)|_{\pi^{-1}\{y\} }\to\pi^{-1}\{y\}$. It is a smoothing operator, so represented by a family of smooth integral kernels $p(t;x,x'):(\wedge T^*Y)\stimes E_x\to(\wedge T^*(Y))\stimes E_{x'}$ (only defined when $\pi(x)=\pi(x')=y$). The Chern character, then, is computed by
\begin{equation}\label{eq:chernasintegral}
(\sCh\pi_!\Sc)(y)=\int_{\pi^{-1}\{y\}}\ud x\,\sTr p(1;x,x).
\end{equation}
\subsection{A families index theorem}
We are now in a position to state and prove the main theorem of this section.
\begin{theorem}\label{th:families}
 Let $\pi:X\to Y$ be a Riemannian family, with compact and spin fibres of dimension $n$. Let $E\to X$ be a finite dimensional, complex, hermitian, $\Zz$-graded vector bundle with unitary superconnection $\Sc$. Then 
\[
 \lim_{t\to0}\sCh\pi_!^t\Sc=(2\pi i)^{-n/2}\pi_*\Ahat(\Omega^\pi)\sCh(\Sc).
\]
\end{theorem}
The $\R^\times$-action on superconnections does not appear explicitly in this statement -- nonetheless it is there. Recall that the $\R^\times$-action on Riemannian maps intertwines with the $\R^\times$-action on superconnections (Eq.~\ref{eq:actionintertwine}). In fact, it is this intertwining that motivates the action on superconnections in the first place.

We will prove a ``local'' version of the theorem. We recall the discussion in Sec.~\ref{sec:chernfam} -- the curvature of the superconnection $\pi_!\Sc$ is a second order elliptic differential operator on the fibres, acting on sections of $\wedge T_y^*Y\stimes(\Sp(X/Y)\stimes E)|_{\pi^{-1}\{y\}}$ for a fibre over $y\in Y$. The Chern character is a smoothing operator on these sections, and thus represented by an integral kernel:
\[
\exp[-t(\pi_!\Sc^s)^2]\phi(\cdot)=\int_{\pi^-1\{y\}}\ud x\,p_{t,s}(\cdot,x)\phi(x).
\]
\begin{proposition}\label{prop:localfamilies}
Let $p_{t,s}(x,y)$ be the integral kernel for $\exp[-t(\pi_!\Sc^s)^2]$. Then there exists a small $t$ asymptotic expansion
\[
 p_{t,s/t}(x,x)\sim(4\pi t)^{-n/2}\sum_{i\ge0}t^iA_i(x),
\]
where
\begin{enumerate}
 \item $A_i(x)\in\sum_{j\le2i}\Forms^j(X;\End_{\Cliff(X/Y)}(E))$,
 \item The full symbol of $p_{t,s/t}(x,x)$, defined by 
\[
\sigma(p)=\sum_{i=0}^{\dim X/2}[\sigma(A_i)]_\bb{2i},
\]
where we recall $\sigma$ on the right is the symbol map defined in Eq.~\ref{eq:cliffsymbol}, is given by
\begin{equation}\label{eq:fullsymbol}
 \sigma(p)=\Ahat(\Omega^\pi)\exp(-\Sc^2).
\end{equation}
\end{enumerate}
\end{proposition}
We note that Th.~\ref{th:localindex} is recovered by setting $Y=\{\mathrm{pt}\}$. The argument in \cite{BGV} Ch.~10 goes through to show Prop.~\ref{prop:localfamilies} implies Th.~\ref{th:families} (the details may be found in \cite{K1}).

The proof of Prop.~\ref{prop:localfamilies} is very similar to that of Th.~\ref{th:localindex}, and we will proceed quickly, exposing the differences, but not dwelling unnecessarily on the details. The general strategy is to follow the proof of Prop.~\ref{th:localindex}, but instead of working on the entire manifold, to work on a point on the fibre, and examine the heat kernel associated to $\pi_!\Sc$ in a neighbourhood of that point on the fibre.

We begin with an analogue of the Weitzenb\"ock formula.
\begin{lemma}\label{lem:weitzenfamily}
 The curvature of the superconnection $\pi_!\Sc$ is given by
\[
 \curv{\pi_!\Sc}=(\nabla^\pi)^*\nabla^\pi+\frac{1}{4}R^\pi+Q_0,
\]
where $(\nabla^\pi)^*\nabla^\pi$ is the covariant laplacian built from the fibre Levi-Civita connection, $R^\pi$ is the fibre scalar curvature, and 
\[
 Q_0=c_0(\Omega^E)+[c_0(\nabla),c_0(\omega)]+c_0(\omega)^2,
\]
 where we have written $\Sc=\nabla+\omega$.
\end{lemma}
\begin{proof}
 We recall that the Bismut superconnection may be thought of the Dirac operator on $X$ (Rem. \ref{rem:alimdirac}), and take the adiabatic limit of the usual Weitzenb\"ock formula. This shows the result.
\end{proof}
Let us now choose $x\in X$, with $y=\pi(x)$. Denote the fibre over $y$ by $X_y$, and choose $r$ such that the exponential map
\[
\exp:B_r(0)\subset T_x(X_y)\to U\subset X_y
\]
is a diffeomorphism of the ball of radius $r$ to an open neighbourhood of $x$. As in the proof of Th.~\ref{th:localindex}, we pull back geometry on $U$ to $T_x(X_y)\cong(T_x(X/Y))$, and extend it to the whole of $T_x(X_y)$. We use the same symbols to denote the geometry on $U$ and $T_x(X_y)$. We choose an orthonormal framing $\{e^i\}$ of $T_x(X_y)$ and a framing $\{f_\alpha\}$ of $T^*(Y)$. Using parallel transport along radial geodesics with respect to the vertical Levi-Civita connection, we may transport the co-frame $\{e_i,f_\alpha\}$ to all of $U$. Similarly, using the parallel transport with respect to the connection $\alim\nabla^X$, we trivialise the bundle $\Cliff_0(X_y)$ and $\Sp_0(X_y)$ on $U$ and may thus identify the degenerate Clifford algebra $\Cliff_0(X)$ with $\bigwedge^\bullet(T^*_xX)$. Care must be taken with the multiplication as $\alim\nabla^X$ is not diagonal, and mixes the $e_i$ and $f_\alpha$. The following lemma (Lemma 10.25 in \cite{BGV}) shows that this mixing is mild.
\begin{lemma}
 The action of the degenerate Clifford algebra $\bigwedge^\bullet(T^*_xX)$ on $\Sp_0(X)\cong\bigwedge^\bullet T^*_yY\stimes \Sp(X/Y)$ (cf. Rem. \ref{rem:clifflim}), denoted by $c_0(\cdot)$ is given by
\begin{eqnarray*}
c_0(e_i)&=&c(e_i)+\sum_\alpha u^\alpha_ie(f_\alpha),\\
c_0(f_\alpha)&=&\extm(f_\alpha), 
\end{eqnarray*}
on $U$, where $c(\cdot)$ is Clifford multiplication on $\Sp(X/Y)$, $\extm(\cdot)$ is exterior multiplication, and the $u^\alpha_i(x')\in\C^\infty(U)$ are $O(|x'|)$.
\end{lemma}
The proof is given in \cite{BGV}, and is a simple computation in the radial coordinates -- the functions $u^\alpha$ come from the fact that $(\alim\nabla^X)e^i\ne0$.

We now introduce an algebraic deformation $\hat{U}_\eps:\bigwedge T^*_x(X)\to \bigwedge T^*_x(X)$ analogous to that in Eq.~\ref{eq:algscale}. It is defined by
\[
 \hat{U}_\eps:\hat{a}\mapsto\eps^{-|\hat{a}|}\hat{a},
\]
on homogeneous elements, and, identifying $\Cliff(X/Y)$ with $\bigwedge T^*(X/Y)$, we extend this to a map $U_\eps:T_y^*(Y)\stimes\Cliff(X_y)\to \bigwedge T^*_x X$. We also introduce the geometric deformation $T_\eps:T(X_x)\to T(X_x)$ as follows
\[
T_\eps:\xi\mapsto\eps\xi.
\]
Defining $S_\eps=U_\eps T^*_\eps$, we see easily that
\begin{equation}\label{eq:slimfam}
 \lim_{\eps\to0}\eps^{|a|} S_\eps c_0(a)S^{-1}_\eps=\extm(\hat{a})
\end{equation}
for homogeneous $a\in\Cliff(X_y)\to \bigwedge T^*_x$. Define the family of heat operators
\begin{eqnarray*}
 H_\eps&=&\frac{\partial}{\partial t}+P_\eps,\\
P_\eps&=&S_\eps\left(\pi_!\Sc^{s/\eps^2}\right)^2S^{-1}_\eps.
\end{eqnarray*}
The following lemma is key.
\begin{lemma}
 The family of operators $P_\eps$ has a limit as $\eps\to0$. In the framing above, it is given by
\[
 P_0=\lim_{\eps\to0}P_\eps=-\sum_{i,j}\left(\frac{\partial}{\partial \xi^i}-\frac{1}{4}e\left(\Omega^{X/Y}_{ij}(x)\right)\xi^j\right)^2+\extm\left((\Sc^s)^2\right)(x),
\]
where the $\xi^i$ are coordinates on $T_x(X_y)$.
\end{lemma}
\begin{proof}
 This lemma is proved entirely analogously to Lemma \ref{lem:heatlim}.
\end{proof}
Following the proof of Th.~\ref{th:localindex}, we next compute the heat kernel for $H_0$.
\begin{proposition}
The heat kernel for the operator
\[
 H_0=\frac{\partial}{\partial t}-\sum_{i,j}\left(\frac{\partial}{\partial \xi^i}-\frac{1}{4}e(\Omega^{X/Y}_{ij}(x)\xi^j\right)^2+\extm\left((\Sc^s)^2\right)(x)
\]
at the origin is given by the formula
\[
 \tilde{p}^0_{t,s}(0,0)=(4\pi t)^{-n/2}\sqrt{\det\left(\frac{t\Omega^\pi\!/2}{\sinh(t\Omega^\pi\!/2)}\right)}\,e^{-t(\Sc)^2}.
\]
\end{proposition}
We note that the small $t$ asymptotic expansion of $p^0_{t,s/t}(0,0)$ is precisely as in Prop.~\ref{prop:localfamilies}.

Let us now define 
\[
\tilde{p}^\eps_{s,t}(x,y)=\eps^nS^\eps p_{\eps^2t,s/\eps^2}(x,y)
\]
and note that this is annihilated by $H_\eps$ and converges to $\delta_x(y)$ as $t\to0$ in $H_k$ for all $k<-n/2$. By uniqueness of solutions this is then the heat kernel for  $H_\eps$. We now argue precisely as in the proof of Th.~\ref{th:localindex}. We write
\[
 p_{t,s}(x,y)=\sum_{I,J}p^{IJ}_{t,s}(x,y)e^If^J,
\]
 where $I$, $J$ are multi-indices. Then
\[
\tilde{p}^\eps_{t,t}(x,s)=\sum_{I,J}\eps^{n-|I|-|J|}p^{IJ}_{\eps^2t,s/\eps^2}(0,0)\hat{e}^If^J.
\]
An easy extension to the argument for a single Dirac operator shows that $p_{t,s/t}(0,0)$ has a small $t$ asymptotic expansion of the form
\[
 p_{t,s/t}(x,x)\sim t^{-n/2}\sum_{j,I,J}t^jA^{IJ}_je^If^J.
\]
Thus
\[
\tilde{p}^\eps_{t,s/t}(x,x)\sim\sum_{j,I,J}t^{j-n/2}\eps^{2j-|I|-|J|}A^{IJ}_j.
\]
We see that $A^{IJ}_j=0$ if $2j<|I|+|J|$, and that 
\[
 \sigma(p)=\sum_{j=0}^{\dim X/2}\sum_{\substack{I,J\\|I|+|J|=2j}}A^{IJ}_{j}=\Ahat(\Omega^\pi)\exp(-(\Sc^s)^2),
\]
as the small $t$ asymptotic expansion of $\tilde{p}^\eps_{t,s/t}$ converges to that of $\tilde{p}^0_{t,s/t}$ as $\eps\to0$.
This then proves the proposition.

There is a useful corollary to Th.~\ref{th:families}. Let $\pi:X\to Y$, $E\to X$ and $\Sc$ be as before. The transgression formula for superconnections shows \cite{BGV}, for $0<t<T$, 
\[
 \sCh\pi_!^T\Sc-\sCh\pi_!^t\Sc=-\ud\int_t^T\ud s\,\sTr\left(\frac{\partial \pi_!^s\Sc}{\partial s}e^{-(\pi_!^s\Sc)^2}\right).
\]
An immediate consequence of Th.~\ref{th:families} is that the left hand side converges as $t\to 0$. The argument in \cite{BGV} Ch.~10.5 goes through to show the following (again, we refer the reader to \cite{K1} for details).
\begin{theorem}\label{th:transgressfamilies}
 \begin{enumerate}
  \item The differential form 
\[
 \alpha(t)=\sTr\left(\frac{\partial \pi_!^t\Sc}{\partial s}e^{-(\pi_!^t\Sc)^2}\right)
\]
has an asymptotic expansion as $t\to 0$ of the form
\[
 \alpha(t)\sim\sum_{j\ge1}t^{j/2-1}\alpha_{j/2},
\]
with $\alpha_{j/2}\in\Forms(B)$.
\item \[\sCh\pi_!^t\Sc=(2\pi i)^{-n/2}\pi_*\Ahat(\Omega^\pi)\sCh\Sc-\ud\int_0^t\ud s\,\alpha(s),\]
where $n=\dim X/Y$.
 \end{enumerate}
\end{theorem}

\section{Determinant line bundles}\label{sec:determinant}
Let $\pi:X\to Y$ be a Riemannian family, with compact and spin fibres. As discussed in the previous section, a complex and hermitian $\Zz$-graded
 vector bundle $E\to X$ with superconnection $\Sc$ gives rise to a family of Dirac operators $\Dirac^\pi(\Sc)$ parameterised by $Y$. In this section we wish to construct the determinant line bundle associated to this family. That is to say, we wish to associate to this data a graded geometric line bundle with a section, $\detline{\pi_!\Sc}\to Y$ in a canonical way, such that the section vanishes precisely when the Dirac operator on the fibre is not invertible. The curvature of the resultant line bundle will be computed by the right hand side of the families index theorem (Th.~\ref{th:families}), i.e. 
\[
\curv{\detline{\pi_!\Sc}}=(2\pi i)^{\frac{1}{2}\dim{X}/Y}\left[\Ahat(\Curv{\pi})\sCh(\Sc)\right]_\bb{2},
\]
and its holonomy by the $\eta$-invariant modulo $\Z$. We proceed as follows: we begin by examining determinant line bundles for finite dimensional $\Zz$-graded vector bundles $V\to X$ with superconnection; we then generalise the construction to give a geometric of the determinant line bundle for families of Dirac operators coupled to superconnections; finally, we compute the holonomy of the resultant line bundle.
\subsection{The finite-dimensional case}\label{ssec:findem}
Let $V\to X$ be a finite-dimensional, complex, and hermitian $\Zz$-graded vector bundle with unitary superconnection $\Sc$ over a smooth manifold $X$. Decompose the superconnection as 
\[
\Sc=\Dirac+\nabla + \sum_{i\ge2}\omega_i,
\]
where $\Dirac\in\End(V)$, $\nabla$ is a connection on $V$, and $\omega_i\in\Forms^i(X;\End(V))$. We decompose
\[
\Dirac=\begin{pmatrix}
&\dslash^*\\
\dslash &
\end{pmatrix}
\]
in terms of the grading on $V$.\footnote{The notation is for easy comparison with the infinite dimensional case, where $\Sc$ will be replaced by the pushed-forward superconnection $\pi_!\Sc$. We recall that the degree zero part of $\pi_!\Sc$ is the fibre Dirac operator.
} Here $\dslash\in\Omega^0(X,\Hom(V^0,\,V^1))$. We will associate to this data, in a functorial manner, a graded geometric line bundle, $\detline{\Sc}\to X$, with a section $\det(\dslash)$ that represents the determinant of $\dslash$, and has curvature $[\sCh(\Sc)]_\bb{2}$.

We define the line bundle
\[
\detline{\Sc}\equiv\bigwedge{}^\mathrm{top} V^0\stimes (\bigwedge{}^\mathrm{top} V^1)^{-1}.
\]
This is a line bundle over $X$, and we grade it by placing it in degree $\dim V^0-\dim V^1$ mod two.

Recalling the canonical isomorphism
\[
\Hom(V^0,V^1)=V^0\stimes (V^1)^*,
\]
we see that $\dslash\in \Gamma(V^0\stimes (V^1)^*)$, so naturally descends to a section $\det(\dslash)\in\detline{\Sc}$. The metric on $V$ also induces a metric on $\detline{\Sc}$.

A connection $\tilde{\nabla}$ is induced on $\detline{\Sc}$ from $\nabla$, the connection part of $\Sc$. Its curvature, however, is computed by $[\sCh(\nabla)]_\bb{2}$ -- we wish to have a connection with curvature computed by the two-form part of $\sCh(\Sc)$. To obtain this we use the transgression formula (Eq.~\ref{eq:transgression}). We begin by considering a general one parameter family of superconnections $\nabla + L_t$, where $\nabla$ is a connection, and $L_t$ is a one parameter family of odd endomorphism valued differential forms. The transgression formula then shows
\begin{equation}\label{eq:dlinetransgress}
 \frac{\ud}{\ud t}\sCh(\nabla+L_t)=-\ud\sTr\left(e^{-(\nabla+L_t)^2}\dot{L}_t\right).
 \end{equation}
Specialising this to two-forms, we see that
\begin{multline}\label{eq:transgressiontwo}
 \frac{\ud}{\ud t}\left[\sCh(\nabla+L_t)\right]_\bb{2}=\ud\bigg\{-\sTr e^{-[L_t]^2_\bb{0}}\left[\dot{L}_t\right]_\bb{1}-\sTr \left[e^{-L_t^2}\right]_\bb{1}\left[\dot{L}_t\right]_\bb{0}\\
+\sTr e^{-[L_t]^2_\bb{0}}\nabla[L_t]_\bb{0}\#e^{-[L_t]^2_\bb{0}}\left[\dot{L}_t\right]_\bb{0}\bigg\},
\end{multline}
where
\[
 e^AB\,\#\,e^C\equiv\int_0^1\ud s\,e^{sA}Be^{(1-s)C}.
\]
We now apply this to the family $\Sc^t$ (where we recall that this is defined in Eq.~\ref{eq:raction}).
\begin{eqnarray}
 \nonumber[\sCh\Sc]_\bb{2}-[\sCh\nabla]_\bb{2}&=&\int_0^1\ud t\,\frac{\ud}{\ud t}[\sCh\Sc^t]_\bb{2}\\
\nonumber&=&\frac{1}{2}\ud\int_0^1\ud t\sTr\left[e^{-t\Dirac^2}\nabla\Dirac\# e^{-t\Dirac^2}\Dirac\right]\\
\nonumber&=&\frac{1}{2}\ud\int_0^1\ud t\sTr\left[\nabla\Dirac e^{-t\Dirac^2}\Dirac\right]\\
\label{eq:curveq}&=&\frac{1}{2}\ud\sTr\left[\left(1-e^{-\Dirac^2}\right)\nabla\Dirac\Dirac^{-1}\right].
\end{eqnarray}
There are several things that may appear problematic in the calculation. The first concern is the possibility of divergences in the first line of the calculation. These can only enter in the terms in the superconnection of cohomological degree greater than or equal to two, so do not contribute.  The second apparent concern is the appearance of $\Dirac^{-1}$: it appears that $\Dirac$ needs to be invertible to make sense of the expression. However, expanding the exponential in Eq.~\ref{eq:curveq} as a power series in $\Dirac$, we see that the leading order term in Eq.~\ref{eq:curveq} is in fact degree one in $\Dirac$, and in this sense the last line is perfectly well defined even when $\Dirac$ is not invertible. 

In the light of the computation \ref{eq:curveq}, we define the connection
\begin{equation}\label{eq:finiteconnection}
 \nabla^{\detline{\Sc}}=\nabla^{\detline{\nabla}}+\frac{1}{2}\sTr\left[\left(1-e^{-\Dirac^2}\right)\nabla\Dirac\Dirac^{-1}\right].
\end{equation}
This is a unitary connection, as $\nabla^{\detline{\nabla}}$ is, and 
\[
 \frac{1}{2}\sTr\left[\left(1-e^{-\Dirac^2}\right)\nabla\Dirac\Dirac^{-1}\right]
\]
is imaginary. It's curvature is computed from $\sCh(\Sc)$:
\[
 \curv{\nabla^{\detline{\Sc}}}=\curv{\nabla^{\detline{\nabla}}}+\frac{1}{2}\ud\sTr\left[\left(1-e^{-\Dirac^2}\right)\nabla\Dirac\Dirac^{-1}\right]=[\sCh\Sc]_\bb{2}.
\]

We have now constructed the determinant line bundle $\detline{\Sc}\to X$, with section, metric and connection. The construction is functorial in the sense that it commutes with pull-backs. The curvature of $\detline{\Sc}$ with the constructed connection is $[\sCh(\Sc)]_\bb{2}$, as desired.
\subsection{Families of Dirac operators} 
We now turn our attention to the situation outlined at the beginning of the section. Let $\pi:X\to Y$ be a Riemannian fibration, with compact and spin fibres of dimension $n$. Let $E\to X$ be a finite-dimensional, complex and hermitian $\Zz$-graded vector bundle, and $\Sc$ a unitary superconnection on $E$. We wish to associate to this data a complex $\Zz$-graded line bundle $\detline{\pi_!\Sc}\to Y$, with section, metric, and connection that should be thought of as the determinant line bundle of the family of fibre Dirac operators $\dslash(\Sc)^\pi$. The determinant line bundle associated to a family of Dirac operators coupled to connections was first constructed by Quillen (for families of $\bar{\partial}$-operators \cite{Q2}) and Bismut and Freed (for families of Dirac operators coupled to ordinary connections \cite{BF1,BF2}). In that case the curvature was computed to be given by the degree two term in the Bismut theorem, and the holonomy by the adiabatic limit of an appropriate $\eta$-invariant mod $\Z$. We will essentially reproduce their construction for general families of Dirac operators coupled to superconnections, with one important modification: for a general superconnection the construction of the connection has to change. Our construction is motivated by the finite dimensional case considered in the previous subsection. In the end the curvature and the holonomy are computed by formulas analogous to those of Bismut and Freed \cite{BF1}.

We proceed as follows: we begin by reviewing the construction of Bismut and Freed, adapting it to our purpose in the obvious fashion. Having constructed a the determinant line bundle, with section and metric, we turn our attention to constructing the connection. Finally we compute the curvature and holonomy of the connection.

We begin by construction the line bundle, with its section. It will be convenient to introduce some notation. We write
\[
 \pi_!\Sc=\Dirac^\pi+\nabla+\sum_{i\ge2}\omega_i,
\]
with $\Dirac^\pi\in\Forms^0(Y;\,\End(\pi_!E))$ is the fibre Dirac operator coupled to $\Sc$,\footnote{Throughout this section we suppress the explicit dependence of the Dirac operator on the superconnection.} $\nabla$ is a connection on $\pi_!E$, and $\omega_i\in\Forms^i(Y;\,\End(\pi_!E))$. We write 
\[
 \Dirac^\pi=\begin{pmatrix}
                   &{\dslash^\pi}^*\\
\dslash^\pi&
                 \end{pmatrix}
\]
so that $\dslash^\pi_y\in\Hom(\pi_!E^0,\pi_!E^1)_y$ at any $y\in Y$.
The spectrum of ${\dslash^\pi_y}^*\dslash^\pi_y$ at any $y\in Y$ is non-negative, and discrete. We define an open cover $\{U_a\}_{a\in\R}$ of $Y$ by
\[
 U_a=\left\{y\in Y:a\not\in\spec{\dslash^\pi_y}^*\dslash^\pi_y\right\},
\]
and define the projections $\Pi_{(a,b)}\in\Gamma(\Hom(\pi_!E))$ which, at each $y\in Y$, projects $(\pi_!E)_y$ to the subspace spanned by eigenvectors with eigenvalues in $(a,\,b)$. We will denote the image of $\Pi_\bb{a,b}$ by $(\pi_!E)_\bb{a,b}$, and allow $a=-\infty$, $b=\infty$. Notice that $\Pi_\bb{a,b}$ respects the grading of $\pi_!E$, and if $b$ is bounded, $(\pi_!E)_{(a,b)}$ is finite dimensional at each $y\in Y$.\footnote{Of course the Laplacians have a non-negative spectrum, so one may replace $\Pi_{(a,b)}$, $a<0$, with $\Pi_{[0,b)}$ without loss of generality.}

Let us now restrict our attention to $U_a$. Over $U_a$ the bundle $\pi_!E$ decomposes as a direct sum
\begin{equation}\label{eq:pushforwarddecomp}
 \pi_!E=(\pi_!E)_\bb{-\infty,a}\oplus(\pi_!E)_\bb{a,\infty},
\end{equation}
and that $\dslash^\pi$ respects the decomposition, restricting to a linear map
\begin{equation}\label{eq:maprestrict}
 \dslash^\pi_\bb{-\infty,a}:(\pi_!E)_\bb{-\infty,a}^\even\to(\pi_!E)_\bb{-\infty,a}^\odd.
\end{equation}
We may thus define the $\Zz$-graded line bundle $\detline{\pi_!\Sc}\to U_a$ as
\[
 \detline{\pi_!\Sc}_\bb{-\infty,a}=\left(\det(\pi_!E)_\bb{-\infty,a}^\odd\right)\stimes\left(\det(\pi_!E)_\bb{-\infty,a}^\even\right)^*,
\]
where ``$\det$'' on the right hand side means the top exterior power.  As in the finite dimensional case,
\[
\dslash^\pi_\bb{-\infty,a}\in\Gamma\left((\pi_!E)_\bb{-\infty,a}^\odd\stimes\left((\pi_!E)_\bb{-\infty,a}^\even\right)^*\right)
\]
induces a section $\Det{\dslash^\pi}_\bb{-\infty,a}\in\Gamma\left(\detline{\pi_!\Sc}_\bb{-\infty,a}\right)$.

We wish to extend the line bundle and section to the whole of $Y$. Let us concentrate on the overlap $U_a\cap U_b$, and assume without loss of generality that $a<b$. On this open set $\pi_!E$ decomposes as
\[
 \pi_!E=(\pi_!E)_\bb{-\infty,a}\oplus(\pi_!E)_\bb{a,b}\oplus(\pi_!E)_\bb{b,\infty},
\]
with
\begin{equation}\label{eq:ubdecomp}
 (\pi_!E)_\bb{-\infty,b}=(\pi_!E)_\bb{-\infty,a}\oplus(\pi_!E)_\bb{a,b}.
\end{equation}
The fibrewise Dirac operator $\dslash^\pi$ respects the decomposition, and is an isomorphism on $(\pi_!E)_\bb{a,b}$. The line bundle $\detline{\pi_!\Sc}_\bb{a,b}$ is purely even, and has a canonical non-vanishing section $\Det{\dslash^\pi }_\bb{a,b}$, where everything is defined analogously to the line bundles in the previous paragraph. In the light of the decomposition Eq.~\ref{eq:ubdecomp}, this gives an isomorphism of $\Zz$-graded line bundles
\begin{eqnarray*}
 \Det{\dslash^\pi}_\bb{a,b}:\detline{\pi_!\Sc}_\bb{-\infty,a}\stackrel{\sim}{\rightarrow}\detline{\pi_!\Sc}_\bb{-\infty,b}
\end{eqnarray*}
defined by
\[
 \sigma\mapsto\sigma\stimes\det\dslash^\pi_\bb{a,b}.
\]
These are the required transition functions, allowing us to define a $\Zz$-graded complex line bundle $\detline{\pi_!\Sc}\to Y$, with section $\Det{\dslash^\pi}$. On any component of $Y$, the degree of the line bundle is given by $\Index\dslash^\pi\mod2$.

We now wish to give $\detline{\pi_!\Sc }$ a hermitian structure. We examine the line bundle on the trivialising neighbourhood $U_a$. The line bundle $\detline{\pi_!\Sc}_\bb{-\infty,a}$ has a natural hermitian structure, induced by the $\Ll$-metric on $\pi_!E$. We compare the metric induced on $\detline{\pi_!(\Sc)}_\bb{-\infty,a}$ with that on $\detline{\pi_!(\Sc)}_\bb{-\infty,b}$. Let $\sigma\in\detline{\pi_!\Sc}$ on the intersection $U_a\cap U_b$. By definition
\[
 \sigma|_b\cong\sigma|_a\stimes\Det{\dslash^\pi}_\bb{a,b},
\]
and thus
\[
 \|\sigma\|^2_\bb{-\infty,b}
=\|\sigma\|^2_\bb{-\infty,a}\prod_{\substack{\lambda\in\spec\dslash^*\dslash\\ a<\lambda<b}}\lambda.
\]
We remove this discrepancy by modifying the metric on $U_a$:
\begin{equation}\label{eq:norm}
 \|\sigma\|^2=\|\sigma\|^2_\bb{-\infty,a}\prod_{\substack{\lambda\in\spec\dslash^*\dslash\\ a<\lambda}}\lambda,
\end{equation}
and similarly on every other open set $U_b$. We may use $\zeta$-function renormalisation to make sense of the formula, as in \cite{BF1}.\footnote{Briefly sketched:

Associated to a second order elliptic operator $E$ one may define a $\zeta$-function
\[
 \zeta[E](s)=\sum_{\lambda\in\spec E}\lambda^s.
\]
This is well defined for $\mathrm{Re}(s)$ large enough, and admits an analytic continuation to a meromorphic function on $\C$. Differentiating the $\zeta$-function in $s$, one sees
\[
 \zeta[E]'(s)=\sum_{\lambda\in\spec E}\log(\lambda)\lambda^{s-1}.
\]
It is thus reasonable to define 
\[
 \prod_{\lambda\in\spec E}\lambda\equiv e^{\zeta[E]'(1)},
\]
as long as $\zeta[E](s)$ is analytic at $s=1$. Bismut and Freed \cite{BF1} show this to be the case for a general first order elliptic differential operator, and it is in this sense that we should interpret Eq.~\ref{eq:norm}. With this interpretation, we may now use Eq.~\ref{eq:norm} to give $\detline{\pi_!\Sc}$  a hermitian structure.}
\subsection{The connection}
We wish to endow $\detline{\pi_!\Sc}\to Y$ with a unitary connection $\nabla^{\detline{\pi_!\Sc}}$, with curvature
\[
 \curv{\nabla^{\detline{\pi_!\Sc}}}=(2\pi i)^{-n/2}\left[\pi_*\Ahat(\Omega^\pi)\sCh\Sc\right]_\bb{2}
\]
where, as usual, $n=\dim X/Y$.
We will do this in two stages: first, we will follow \cite{BF1} to construct a connection that may be seen as that induced by $\nabla^{\pi_!E}$. We will then use similar arguments to those section \ref{ssec:findem} to modify this connection and produce the connection we require.

To begin we restrict our attention to $U_a$. The connection $\nabla^{\pi_!E}$ restricts to a connection $\nabla^{\pi_!E}_\bb{-\infty,a}$ on $\detline{\pi_!\Sc}_\bb{-\infty,a}$. It is unitary with respect to the induced metric, but not with respect to the metric on $\detline{\pi_!\Sc}$. Let $\sigma\in\Gamma(\detline{\pi_!\Sc})$.  We calculate on the overlap $U_a\cap U_b$:
\begin{eqnarray*}
 \nabla^{\pi_!E}_\bb{-\infty,b}\sigma|_b&=&\nabla^{\pi_!E}_\bb{-\infty,b}\left(\sigma|_a\stimes\Det{\dslash^\pi}_\bb{a,b}\right)\\
&=&\left(\nabla^{\pi_!E}_\bb{-\infty,a}\sigma|_a\right)\stimes\Det{\dslash^\pi}_\bb{a,b}\\
&&\qquad\qquad+\Tr\left[(\nabla\dslash)\dslash^{-1}\right]_\bb{a,b}\sigma|_a\stimes{\Det{\dslash^\pi}}_\bb{a,b},
\end{eqnarray*}
where $\Tr$ is the \emph{ungraded} trace.

We follow Bismut and Freed, and define a connection on $\detline{\pi_!\Sc}$ by
\begin{equation}\label{eq:stagezero}
 \nabla^0_\bb{-\infty,a}=\nabla^{\pi_!E}_\bb{-\infty,a}+\alpha_a
\end{equation}
on $U_a$, where 
\[
\alpha_a=\Tr\left[(\nabla\dslash)\dslash^{-1}\right]_\bb{a,\infty}.
\]
Here $\alpha_a$ is defined as in \cite{BF1} using $\zeta$-function renormalisation, and the resulting connection is unitary and has curvature
\[
 \curv{\nabla^0}=\LIM_{t\to0}\left[\sCh{(\pi_!\Sc)^t}\right]_\bb{2}.
\]
Here ``$\LIM$" is defined as follows. Suppose $f(t)\in C^\infty((0,\infty))$, with a small $t$ asymptotic expansion
\[
f(t)\sim\sum_{k\ge k_0} a_k t^{k}.
\]
Then we define
\begin{equation}\label{eq:LIMdef}
\LIM_{t\to0}f(t)\equiv a_0.
\end{equation}

We will now modify the connection in Eq.~\ref{eq:stagezero} to obtain the connection we desire. We will follow an argument similar to that in  Sec.~\ref{ssec:findem}. We begin by introducing two ancillary one-forms:
\begin{eqnarray*}
 \beta_1(t)&=&\int_t^1\ud t\,\left[\sTr\left(e^{(\pi_!\Sc)^t}\frac{\ud(\pi_!\Sc)^t}{\ud t}\right)\right]_\bb{1},\\
 \beta_2(t)&=&\int_t^1\ud t\,\left[\sTr\left(e^{\pi^t_!\Sc}\frac{\ud \pi^t_!\Sc}{\ud t}\right)\right]_\bb{1}.
\end{eqnarray*}
A calculation similar to that in Eq.~\ref{eq:curveq} gives a closed expression for $\beta_1(t)$:
\[
 \beta_1(t)=\frac{1}{2}\sTr\left[e^{-\Dirac(\Sc^t)^2}\nabla\Dirac(\Sc)\Dirac^{-1}(\Sc)\right]_\bb{2}
-\frac{1}{2}\sTr\left[e^{-\Dirac(\Sc)^2}\nabla\Dirac(\Sc)\Dirac^{-1}(\Sc)\right]_\bb{2}.
\]
Unfortunately $\beta_2(t)$ may not be written so explicitly.
The transgression formula shows immediately that
\begin{eqnarray*}
\ud\beta_1(t)&=&\left[\sCh\pi_!\Sc\right]_\bb{2}-\left[\sCh(\pi_!\Sc)^t\right]_\bb{2},\\
\ud\beta_2(t)&=&\left[\sCh\pi_!\Sc\right]_\bb{2}-\left[\sCh(\pi^t_!\Sc)\right]_\bb{2}.
\end{eqnarray*}
Both of these one-forms are purely imaginary. We define the connection
\begin{equation}\label{eq:connectiondetline}
 \nabla^{\detline{\pi_!\Sc}}\equiv\nabla^0+\LIM_{t\to0}\beta_1(t)-\lim_{t\to0}\beta_2(t),
\end{equation}
where we recall $\LIM$ is the formal limit defined in Eq.~\ref{eq:LIMdef}, and the limit
$
\lim_{t\to0}\beta_2(t)
$
converges by virtue of Th.~\ref{th:transgressfamilies}.
By virtue of the unitarity of $\nabla^0$, the connection $\nabla^{\detline{\pi_!\Sc}}$ is a unitary connection, and the families index theorem (Th.~\ref{th:families}) shows that it has curvature
\begin{equation}\label{eq:linecurv}
 \curv{\nabla^{\detline{\pi_!\Sc}}}=(2\pi i)^{-n/2}\left[\pi_*\Ahat(\Omega^\pi)\sCh\Sc\right]_\bb{2},
\end{equation}
where $n=\dim X/Y$.
\subsection{The holonomy of the determinant line bundle}
We now calculate the holonomy of the determinant line bundle.
\begin{theorem}\label{th:holonomy}
 Let $X\to Y$ be a Riemannian family with compact and spin fibres, and $E\to X$ be a finite dimensional, complex, and hermitian $\Zz$-graded vector bundle with superconnection $\Sc$. The holonomy of the associated determinant line bundle is computed by the adiabatic limit of the $\tau$-invariant. More precisely, let $\gamma:[0,1]\to Y$ be a smooth loop (i.e. the map and all of its derivatives agree at the end points). Then
\[
 \Hol{\gamma}{\detline{\pi_!\Sc}}=\alim\tau_{\pi^{-1}\gamma}^{-1}(\Sc),
\]
where the $\tau$-invariant may be computed with any choice of metric on the base.
\end{theorem}
\begin{proof}
 We will construct a mapping cylinder relating our situation to the analogous situation where the superconnection is an ordinary connection, and then use the theorem of Bismut and Freed \cite{BF2} for the holonomy of determinant line bundles of families of Dirac operators coupled to connections.

Without loss of generality, we may restrict our attention to Riemannian maps $\pi:X\to S^1$, with the standard parameter $\theta$ and metric on $S^1$. We endow the circle with the bounding spin structure. We form the associated family $\tilde{\pi}:X\times[0,1]\to S^1\times[0,1]$, with the $\Zz$-graded vector bundle $E\times[0,1]\to X\times[0,1]$. Writing $t$ as the parameter in $[0,1]$, we endow $E\times[0,1]$ with the superconnection $\tilde{\Sc}=(1-t)\nabla+t\Sc$, where we have written $\Sc=\nabla+\omega$. By Stokes' theorem
\[
 2\pi i \int_{S^1\times[0,1]}\curv{\nabla^{\detline{\pi_!\Sc}}}=\log\Hol{S^1}{\detline{\pi_!\Sc}}-\log\Hol{S^1}{\detline{\pi_!\nabla}}.
\]
We recall
\[
\int_{S^1\times[0,1]}\curv{\nabla^{\detline{\pi_!\Sc}}}=(2\pi i)^{-n/2}\int_{S^1\times[0,1]}\pi_*\Ahat(\Omega^{\tilde{\pi}})\sCh\Sc.
\]
On the other hand, applying the APS theorem we see that
\begin{multline*}
(2\pi i)^{-1}\left(\log\tau_X(\Sc)-\log\tau_X(\nabla)\right)\\
+(2\pi i)^{-n/2}\int_{X\times[0,1]}\Ahat\left(\Omega^{X\times[0,1]}\right)\sCh\Sc=0\mod1.
\end{multline*}
But
\begin{eqnarray*}
 \alim\int_{X\times[0,1]}\Ahat\left(\Omega^{X\times[0,1]}\right)\sCh\Sc&=&\int_{S^1\times[0,1]}\!\!\Ahat\left(\Omega^{S^1\times[0,1]}\right)\,\pi_*\Ahat\left(\Omega^{Y\times[0,1]}\right)\sCh\Sc\\
&=&\int_{S^1\times[0,1]}\!\!\pi_*\Ahat\left(\Omega^{Y\times[0,1]}\right)\sCh\Sc.
\end{eqnarray*}
But then
\[
 \left(\Hol{S^1}{\Sc}\right)\left(\alim\tau_X(\Sc)\right)= \left(\Hol{S^1}{\nabla}\right)\left(\alim\tau_X(\nabla)\right).
\]
Applying the theorem of Bismut and Freed \cite{BF2} yields the result.
\end{proof}
\subsection{Summary}
We have now produced a complex $\Zz$-graded line bundle $\detline{\pi_!\Sc}\to Y$ with section, metric (defined by Eq.~\ref{eq:norm}) and connection $\nabla^{\detline{\pi_!\Sc}}$ (defined by Eq.~\ref{eq:connectiondetline}) associated to the data of a Riemannian family with compact and spin fibres $\pi:X\to Y$, and a finite dimensional, complex and hermitian $\Zz$-graded vector bundle $E\to X$ with unitary superconnection $\Sc$. The curvature of the line bundle is computed by right hand side of the families index theorem (Eq.~\ref{eq:linecurv}), and the holonomy is computed by an appropriate $\eta$-invariant. Furthermore, it is easy to see that the construction is natural in the sense that commutes with pullbacks.
\bibliographystyle{amsplain}
\bibliography{sc}
\end{document}